\documentclass[11pt,oneside,english]{amsart}
\usepackage[T1]{fontenc}
\usepackage[utf8]{inputenc}
\usepackage{color}
\usepackage{amsthm}
\usepackage{amstext}
\usepackage{amssymb}
\usepackage{esint}
\makeatletter
\numberwithin{equation}{section}
\numberwithin{figure}{section}
\theoremstyle{plain}
\newtheorem{thm}{\protect\theoremname}[section]
  \theoremstyle{definition}
  \newtheorem{defn}{\protect\definitionname}[section]
  \theoremstyle{plain}
  \newtheorem{lem}{\protect\lemmaname}[section]
  \theoremstyle{plain}
  \newtheorem{cor}{\protect\corollaryname}[section]
  \theoremstyle{remark}
  \newtheorem{ex}{\protect\examplename}[section]
  \theoremstyle{remark}
  \newtheorem{rem}{\protect\remarkname}[section]
  \theoremstyle{plain}
  \newtheorem{assumption}{\protect\assumptionname}[section]
  \theoremstyle{plain}
  \newtheorem{prop}{\protect\propositionname}[section]
  
\usepackage{appendix}

\makeatother

\usepackage{babel}
  \providecommand{\assumptionname}{Assumption}
  \providecommand{\corollaryname}{Corollary}
  \providecommand{\definitionname}{Definition}
  \providecommand{\lemmaname}{Lemma}
  \providecommand{\propositionname}{Proposition}
  \providecommand{\remarkname}{Remark}
  \providecommand{\examplename}{Example}
\providecommand{\theoremname}{Theorem}

\begin{document}
\title{Brownian Bridges on Random Intervals}
\author{M. L. Bedini, R. Buckdahn, H.-J. Engelbert}
\date{July 13, 2015}
\begin{abstract}
The issue of giving an explicit description of the flow of information
concerning the time of bankruptcy of a company (or a state) arriving
on the market is tackled by defining a bridge process starting from
zero and conditioned to be equal to zero when the default occurs.
This enables to catch some empirical facts on the behavior of financial markets: when the bridge process is away from zero, investors can be relatively sure that the default will not happen immediately. However, when the information process is close to zero, market agents should be aware of the risk of an imminent default. In this sense the bridge process leaks information
concerning the default before it occurs. The objective of this first
paper on Brownian bridges on stochastic intervals is to provide the
basic properties of these processes.
\end{abstract}


\maketitle 

\section{\label{sec:Introduction}Introduction}
\noindent Motivated by the problem of modeling the information concerning the
default time of a financial company, we propose a new approach to
credit-risk. The problem is to give a mathematical model for the flow
of information about the time at which the bankruptcy of a company
(or state) occurs. The time at which this crucial event takes place
is called \textit{default time} and it is modeled by a strictly positive
random variable $\tau$. We tackle the problem by defining the process
$\beta=(\beta_{t},\, t\geq0)$, a Brownian bridge between
0 and 0 on the random time interval $[0,\tau]$:
\begin{equation}
\beta_{t}:=W_{t}-\frac{t}{\tau\vee t}W_{\tau\vee t},\quad t\geq 0\,, \label{eq:EQ0}
\end{equation}
where $W=(W_{t},\, t\geq0)$ is a Brownian motion independent
of $\tau$. Since we are going to model the information about $\tau$
with such a bridge process, we call $\beta$ the \textit{information
process}. The filtration $\mathbb{F}^{\beta}=(\mathcal{F}_{t}^{\beta})_{t\geq0}$
generated by the information process provides partial information
on the default time before it occurs. The intuitive idea is that when
market agents observe $\beta_{t}$ away from 0, they know that the
default will not occur immediately; on the other hand, the period
of fear of an imminent default corresponds to the situation in which
$\beta_{t}$ is close to 0. In our approach the Bayes formula is of
particular importance since it allows to infer the a posteriori probability
distribution of the default time $\tau$ given the information carried
by $\mathcal{F}_{t}^{\beta}$, for $t\geq0$.

Two classes of credit risk models have greatly inspired our work:
the information-based approach and the reduced-form models. The first
one has been developed in \cite{key-20} by Brody, Hughston and Macrina.
They model the information concerning a random payoff $D_{T}$, paid
at some pre-established date $T>0$, with the natural, completed filtration
generated by the process $\xi=(\xi_{t},\,0\leq t\leq T)$,
defined by $\xi_{t}:=\beta_{t}^{T}+\alpha tD_{T}$, where $\beta^{T}=(\beta_{t}^{T},\,0\leq t\leq T)$
is a standard Brownian bridge on the deterministic time interval $[0,T]$.
The process $\beta^{T}$ is independent of $D_{T}$, and $\alpha>0$
is a positive constant. The idea is that the true information, represented
by the component $\alpha tD_{T}$, about the final payoff is disturbed
by some noisy information (rumors, mispricing, etc.), represented
by the bridge process. In this model for credit risk the information
concerning the random cash-flow $D_{T}$ is modeled explicitly but
the default time $\tau$ of the company is not. On the other hand,
in the models following the reduced-form approach for credit risk
the information on the default time is modeled by the natural, completed
filtration generated by the single-jump process $H=(H_{t}:=\mathbb{I}_{\{ \tau\leq t\} },\, t\geq0)$
occurring at $\tau$. We refer to the book 
\cite{key-21} of Bielecki and Rutkowski and the series of papers \cite{key-23,key-24,key-25} of Jeanblanc and Le Cam,
among many other works on the reduced-form approach to credit-risk.
Besides the advantages that have made this approach well-known, there
is the poor structure of the information concerning $\tau$: people
just know if the default has occurred ($H_{t}=1$) or not ($H_{t}=0$).
Financial reality is often more complex than this: market agents have
indeed more information and there are actually periods in which the
default is more likely to happen than in others. We try to reconcile
those two approaches considering that in our model the information is carried by
the process $\beta$ given by (\ref{eq:EQ0}).

This paper is organized as follows. In Section 2 we provide some preliminary facts that are used throughout the paper. In Section 3 we give the definition of a bridge $\beta$ of random length $\tau$. We show that $\tau$ is a stopping time with
respect to the natural, completed filtration generated by the process
$\beta$. Also, we establish that $\beta$ is a Markov
process with respect to the filtration $\mathbb{F}^{0}:=(\sigma(\beta_{s},\,0\leq s\leq t))_{t\geq0}$.
In Section 4 we derive Bayesian estimates of the distribution of the default time $\tau$. Thereafter, in Section 5 we extend the results obtained in the previous section to more general situations. Section 6 is devoted to the proof of the
Markov property of the information process with respect to the
filtration $\mathbb{F}^{\beta}$, the smallest filtration which contains $\mathbb{F}^0$ and satisfies the usual conditions of right-continuity and completeness. In Section 7 we show that the process $\beta$ is a semimartingale and we provide its decomposition. Finally, in Section 8 we consider an application to Mathematical Finance concerned with the problem of pricing a Credit Default Swap in an elementary market model. The paper closes with the Appendix where, for the sake of easy reference, we recall the Bayes Theorem, and we prove a slight extension of the so-called innovation lemma (used in Section \ref{sec:Semimartingale-decomposition} for the semimartingale decomposition of the information process $\beta$) and, finally,  an auxiliary result. 

This paper is based on the thesis \cite{key-10}
and the main objective of it is to introduce and study the information process $\beta$ and to provide its basic properties. Other topics of the thesis \cite{key-10} are the enlargement of the filtration $\mathbb{F}^{\beta}$ with a reference filtration $\mathbb{F}$, the classification of the default time $\tau$ (with respect to the filtration $\mathbb{F}^{\beta}$ and with respect to the enlarged filtration) and further applications to Mathematical Finance.

\section{Preliminaries}
\noindent We start by recalling some basic results on Brownian bridges and properties of conditional expectations that will be used in the sequel. The interested reader may find further discussions on Brownian bridges, among others, in the books \cite{key-1} of Karatzas and Shreve or \cite{key-3} of Revuz and Yor . 

As usual, the set of natural numbers is denoted by $\mathbb{N}$ and the set of real numbers by $\mathbb{R}$. If $A\subseteq\mathbb{R}$, then the notation $A_{+}$ stands for 
$A_{+}:=A\cap\{ x\in\mathbb{R}:x\geq0\}$.
If $E$ is a topological space, then the Borel $\sigma$-algebra over $E$ will be denoted by $\mathcal{B}(E)$. If $A$ is a set, its indicator function will be denoted by $\mathbb{I}_{A}$.

Let $(\Omega,\,\mathcal{F},\,\mathbf{P})$ be a complete
probability space. By $\mathcal{N}_{P}$ we denote the collection
of $\mathbf{P}$-null sets. If $X$ is a random
variable, the symbol $\sigma(X)$ will denote the $\sigma$-algebra generated by $X$. 

Let $W:(\Omega,\mathcal{F})\rightarrow(C,\mathcal{C})$
be a map from $\Omega$ into the space of continuous real-valued functions
defined on $\mathbb{R}_{+}$ endowed with the $\sigma$-algebra generated
by the canonical process, so that with each $\omega\in\Omega$ we associate the continuous function $W(\omega)=(W_{t}(\omega),\, t\geq0)$. 
\begin{defn}
Given a strictly positive real number $r$,
the function $\beta^{r}:\Omega\rightarrow C$ defined by 
\begin{equation*}
\beta_{t}^{r}\left(\omega\right):=W_{t}\left(\omega\right)-\frac{t}{r\vee t}W_{r\vee t}\left(\omega\right), \ t\geq0, \ \omega\in\Omega\,,
\end{equation*}
is called \textit{bridge of length $r$ associated with $W$}. If $W$ is a Brownian motion on the probability space $(\Omega,\mathcal{F},\mathbf{P})$, then the process
$\beta^{r}$ is called \textit{Brownian bridge of length $r$}.
\end{defn}
We have the following fact concerning the measurability of the process $\beta^{r}$.
\begin{lem}
The map $(r,t,\omega)\mapsto\beta_{t}^{r}(\omega)$ of $((0,+\infty)\times\mathbb{R}_{+}\times\Omega,\,\mathcal{B}((0,+\infty))\otimes\mathcal{B}(\mathbb{R}_{+})\otimes\mathcal{F})$
into $(\mathbb{R},\,\mathcal{B}(\mathbb{R}))$ is measurable. In particular, the $t$-section of $(r,t,\omega)\rightarrow\beta_{t}^{r}(\omega)$: $(r,\omega)\mapsto\beta_{t}^{r}(\omega)$
is measurable with respect to the $\sigma$-algebra $\mathcal{B}((0,+\infty))\otimes\mathcal{F}$, for all $t\geq0$.\end{lem}
\begin{proof}
From the definition of $\beta^{r}$ we have that
(i) the map $\omega\rightarrow\beta_{t}^{r}(\omega)$ is measurable for all $r>0$ and $t\geq0$ and (ii) the map $(r,t)\rightarrow\beta_{t}^{r}(\omega)$
is continuous for all $\omega\in\Omega$. It now suffices to proceed with the discretization of the parameter $(r,t)$ in order to define piecewise constant and measurable
functions converging pointwise to $(r,t,\omega)\mapsto\beta_{t}^{r}(\omega)$
and to use standard results on the passage to the limit of sequences of measurable functions. The assertion of
the lemma then follows immediately.\end{proof}
\begin{cor}
\label{cor:measyBetaR} The map $(r,\omega)\rightarrow\beta_{\cdot}^{r}(\omega)$ of $((0,+\infty)\times\Omega,\,\mathcal{B}((0,+\infty))\otimes\mathcal{F})$
into $(C,\,\mathcal{C})$ is measurable.\end{cor}
\begin{rem}
Note that $\beta^{r}$ is just as the $r$-section of the map
$(r,t,\omega)\rightarrow\beta_{t}^{r}(\omega)$.
\end{rem}
If the process $W$ is a Brownian motion on $(\Omega,\mathcal{F},\mathbf{P})$, the process $\beta^{r}$ is just a Brownian bridge which is identically
equal to 0 on the time interval $[r,\,+\infty)$. If we denote by $p(t,x,y),\, x\in\mathbb{R}$, the Gaussian density with variance $t$ and mean $y$, then the function $\varphi_{t}(r,\cdot)$ given by
\begin{equation}
\varphi_{t}\left(r,x\right):=\begin{cases}
p\left(\frac{t\left(r-t\right)}{r},x,0\right), & 0<t<r, \ x\in\mathbb{R},\\
0, & r\leq t, \ x\in\mathbb{R},
\end{cases} \label{eq:bbdensity}
\end{equation}
 is equal to the density of $\beta_{t}^{r}$ for $0<t<r$. Furthermore
we have the following properties:
\begin{enumerate}
\item[(i)] (Expectation) \ $\mathbf{E}\left[\beta_{t}^{r}\right]=0, \ t\geq0.$
\item[(ii)] (Covariance) \ $\mathbf{E}\left[\beta_{t}^{r}\beta_{s}^{r}\right]=s\wedge t\wedge r-\frac{(s\wedge r)(t\wedge r)}{r}, \ s,t\geq 0$.
\item[(iii)] (Conditional expectation) \  
For $t\in[0,r)$ and $t<u$ we have: If $g(\beta^r _u)$ is integrable, then
\begin{eqnarray}
\lefteqn{\mathbf{E}\left[g\left(\beta_{u}^{r}\right)|\beta_{t}^{r}=x\right]}\nonumber\\
&=&\begin{cases}
{\displaystyle \int_{\mathbb{R}}}g\left(y\right)p\left(\frac{r-u}{r-t}\left(u-t\right),\, y,\,\frac{r-u}{r-t}x\right)dy, & t<u<r, \ x\in\mathbb{R},\\
g\left(0\right), & r\leq u, \ x\in\mathbb{R}.
\end{cases}\label{eq:cond exp ext bb}
\end{eqnarray}
\item[(iv)] $\beta^{r}$ is a Markov process with respect to its completed natural filtration.
\item[(v)] \ $\beta_{t}^{r}=\beta_{t\wedge r}^{r}, \ t\geq0$.
\end{enumerate}
\begin{lem}[{\cite[(4.3.8)]{key-4}}]\label{BMr}
The process $B^{r}=(B_{t}^{r},\, t\geq0)$ given by
\begin{equation}
B_{t}^{r}:=\beta_{t}^{r}+\int_{0}^{t}\frac{\beta_{s}^{r}}{r-s}\, ds\label{eq:BMfrom_extBB}
\end{equation}
is a Brownian motion stopped at the (deterministic) stopping time
$r$, with respect to the completed natural filtration
of $\beta^{r}$.
\end{lem}

Now we substitute the fixed time $r$ by a random time
$\tau$. Let $\tau$ be a strictly positive random variable on $(\Omega,\mathcal{F},\mathbf{P})$. The object which we are
interested in is given by the composition of the two mappings $(r,t,\omega)\mapsto\beta_{t}^{r}(\omega)$
and $(t,\omega)\mapsto(\tau(\omega),t,\omega)$. We have the following definition:
\begin{defn}
The map $\beta=(\beta_{t},\, t\geq0):\,(\Omega,\,\mathcal{F})\rightarrow(C,\,\mathcal{C})$
is defined by  
\begin{equation}
\beta_{t}\left(\omega\right):=\beta_{t}^{\tau\left(\omega\right)}\left(\omega\right),\quad (t,\omega)\in\mathbb{R}_+\times\Omega\,.\label{eq:beta def 0}
\end{equation}
\end{defn}
\begin{lem}
The map $\beta:\,(\Omega,\,\mathcal{F})\rightarrow(C,\,\mathcal{C})$
is measurable.
\end{lem}
\begin{proof}
The map $(r,\omega)\rightarrow\beta^{r}(\omega)$
of $((0,+\infty)\times\Omega,\,\mathcal{B}((0,+ \infty))\otimes\mathcal{F})$
into $(C,\,\mathcal{C})$ is measurable because of Corollary
\ref{cor:measyBetaR}. By definition, the map $\omega\rightarrow(\tau(\omega),\omega)$
of $(\Omega,\mathcal{F})$ into the measurable space $((0,+\infty)\times\Omega,\,\mathcal{B}((0,+\infty))\otimes\mathcal{F})$ is measurable. The statement of the lemma follows from the measurability of the composition of measurable functions. 
\end{proof}
We now consider the conditional law with respect to $\tau$. As common, by $\mathbf{P}_\tau$ we denote the distribution of $\tau$ with respect to $\mathbf{P}$.
\begin{lem}
\label{lem:If--is}If $\tau$ is independent of
the Brownian motion $W$, then for any measurable function $G$ on $\left(\left(0,+\infty\right)\times C,\mathcal{B}\left(\left(0,+\infty\right)\right)\otimes\mathcal{C}\right)$ such that $G(\tau,\, W)$ is integrable it follows that
\[
\mathbf{E}\left[G\left(\tau,\, W\right)|\sigma\left(\tau\right)\right]\left(\omega\right)=\left(\mathbf{E}\left[G\left(r,W\right)\right]\right)_{r=\tau\left(\omega\right)},\;\mathbf{P}\textrm{-a.s.}
\]
\end{lem}
\begin{proof} See, e.g., \cite[Ch. II.7, p. 221]{key-8}.
\end{proof}
\begin{cor}
\label{cor:LEM if--is} If $h:\,\left((0,+\infty)\times C,\mathcal{B}\left(\left(0,+\infty\right)\right)\otimes\mathcal{C}\right)\mapsto\left(\mathbb{R},\mathcal{B}(\mathbb{R})\right)$
is a measurable function such that $\mathbf{E}\left[|h\left(\tau,\,\beta\right)|\right]<+\infty$,
then 
\[
\mathbf{E}\left[h\left(\tau,\,\beta\right)|\tau=r\right]=\mathbf{E}\left[h\left(r,\,\beta^{r}\right)\right],\; \mathbf{P}_{\tau}\textrm{-a.s.},
\]
or, equivalently,
\[
\mathbf{E}\left[h\left(\tau,\,\beta\right)|\sigma\left(\tau\right)\right]\left(\omega\right)=\left(\mathbf{E}\left[h\left(r,\,\beta^{r}\right)\right]\right)_{r=\tau\left(\omega\right)},\; \mathbf{P}\textrm{-a.s.}
\]
\end{cor}
The last two formulas provide a useful connection between the law of Brownian bridges and the conditional law with respect to $\sigma(\tau)$ of a generic functional involving $\tau$, the Brownian motion $W$ and the map $\beta$ defined in (\ref{eq:beta def 0}).

\section{\label{sec:Definition-and-First} Definition and First Properties} 
\noindent Let $W=(W_{t},\, t\geq0)$ be a standard Brownian motion (with respect to its completed natural filtration $\mathbb{F}^{W}=(\mathcal{F}_{t}^{W})_{t\geq0}$)
starting from 0. Let $\tau:\,\Omega\rightarrow(0,\infty)$
be a strictly positive random time with distribution function denoted
by $F$: $F(t):=\mathbf{P}(\tau\leq t),\: t\geq0$. In
this paper the following assumption will always be made.
\begin{assumption}
\label{ass:assumption1}The random time $\tau$ and the Brownian motion
$W$ are independent.
\end{assumption}
The random time $\tau$ (which will be interpreted as \textit{default
time}) is supposed to be not an $\mathbb{F}^{W}$-stopping time. This
means that the case in which $\tau$ is equal to a
positive constant $T$ will not be considered.
\begin{defn}
\label{def:obs by MonPon}The process $\beta=(\beta_{t},\;t\geq0)$
given by
\begin{equation}
\beta_{t}:=W_{t}-\frac{t}{\tau\vee t}W_{\tau\vee t},\; t\geq0,\label{eq:betaDEF1}
\end{equation}
 will be called \textit{information process}.\end{defn}
In what follows, we shall use the following filtrations:
\begin{enumerate}
\item[(i)] $\mathbb{F}^{0}=(\mathcal{F}_{t}^{0})_{t\geq0}$, the natural
filtration of the process $\beta$:\\ $\mathcal{F}_{t}^{0}:=\sigma(\beta_{s},\,0\leq s\leq t),\ t\geq0$.
\item[(ii)] $\mathbb{F}^{P}=(\mathcal{F}_{t}^{P})_{t\geq0}$, the completed natural filtration:\\
$\mathcal{F}_{t}^{P}:=\mathcal{F}_{t}^{0}\vee\mathcal{N}_{P},\, t\geq0$.
\item[(iii)] $\mathbb{F}^{\beta}=(\mathcal{F}_{t}^{\beta})_{t\geq0}$,
the smallest filtration containing $\mathbb{F}^{0}$ and satisfying
the usual hypotheses of right-continuity and completeness:
\[
\mathcal{F}_{t}^{\beta}:=\sigma\left(\bigcap_{u>t}\mathcal{F}_{u}^{0}\cup\mathcal{N}_{P}\right)=\mathcal{F}_{t+}^{0}\vee\mathcal{N}_{P},\: t\geq0.
\]
\end{enumerate}
The aim of this section is to prove that the default time $\tau$ is a stopping time with respect to the completed natural filtration $\mathbb{F}^P$ and that the information process $\beta$ is Markov with respect to $\mathbb{F}^P$. 

Although it is possible to prove directly that $\tau$ is an $\mathbb{F}^{\beta}$-stopping time (see \cite[Lemma 2.5]{key-10}) we point out the following result which better involves the role of the probability measure $\mathbf{P}$. For two sets $A, B\in\mathcal{F}$ we shall write $A\subseteq B \ \; \mathbf{P}$-a.s. if $\mathbf{P}(B\setminus A)=0$. If $A\subseteq B \ \; \mathbf{P}$-a.s. and $B\subseteq A \ \; \mathbf{P}$-a.s., then we write $A=B \ \; \mathbf{P}$-a.s.
\begin{prop}
\label{lem:For-all-,}For all $t>0$, $\{ \beta_{t}=0\} =\{ \tau\leq t\},\ \mathbf{P}$-a.s.
In particular, $\tau$ is a stopping time with respect to the filtration
$\mathbb{F}^{P}$.
\end{prop}
\begin{proof}
Using the formula of total probability and Corollary
\ref{cor:LEM if--is}, we have that 
\begin{align*}
\mathbf{P}\left(\beta_{t}=0,\,\tau>t\right) &=\int_{\left(t,+\infty\right)}\mathbf{P}\left(\beta_{t}=0|\tau=r\right)dF(r)\\
&=\int_{\left(t,+\infty\right)}\mathbf{P}\left(\beta_{t}^{r}=0\right)dF\left(r\right)=0\,,
\end{align*}
where the latter equality holds because the random variable $\beta_{t}^{r}$ is nondegenerate and normally distributed for $0<t<r$ and, therefore, we obtain $\mathbf{P}(\beta_{t}^{r}=0)=0,\,0<t<r$.
Hence, for all $t>0$, $\{ \beta_{t}=0\} \cap\{ \tau>t\} \in\mathcal{N}_{P}$ and, consequently, $\{ \beta_{t}=0\} \subseteq\{ \tau\leq t\} ,\,\mathbf{P}$-a.s.
In view of $\{ \tau\leq t\} \subseteq\{ \beta_{t}=0\}$,
this yields the first part of the statement: $\{ \tau\leq t\} =\{\beta_{t}=0\} ,\,\mathbf{P}$-a.s.
Since $\{ \beta_{t}=0\} \in\mathcal{F}_{t}^{0}$, the event
$\{ \tau\leq t\}$ belongs to $\mathcal{F}_{t}^{0}\vee\mathcal{N}_{P}$,
for all $t\geq 0$. Hence $\tau$ is a stopping
time with respect to $\mathbb{F}^{P}$.
\end{proof}
\begin{cor}
$\tau$ is a stopping time with respect to the filtration $\mathbb{F}^{\beta}$.
\end{cor}
The above proposition states that the process $(\mathbb{I}_{\{ \tau\leq t\} },\, t>0)$ is a modification, under the probability measure $\mathbf{P}$, of the process $(\mathbb{I}_{\{ \beta_{t}=0\} },\, t>0)$. On the other hand, these both processes are not indistinguishable, since the process $\beta$ can hit 0 before $\tau$ (in fact, due to the law of iterated logarithm, this happens
uncountably many times $\mathbf{P}$-a.s.). 

Roughly speaking we can say that in general, if we can observe only
$\beta_{t}$, we are sure that $\tau>t$ whenever $\beta_{t}\neq0$.
But, if at time $t$ we observe the event $\{ \beta_{t}=0\}$,
the information carried by $\mathcal{F}_{t}^{0}$ may be not sufficient
to know whether the event $\{ \tau\le t\}$ has occurred
or not, we can only say that it occurred $\mathbf{P}$-a.s. 

We now show that the information process $\beta$ is a Markov process with
respect to its natural filtration $\mathbb{F}^{0}$. 
\begin{thm}
\label{TEO:The-information-process} The information process $\beta$
is an $\mathbb{F}^{0}$-Markov process:
\begin{equation}
\mathbf{E}\left[f\left(\beta_{t+h}\right)|\mathcal{F}_{t}^{0}\right]=\mathbf{E}\left[f\left(\beta_{t+h}\right)|\beta_{t}\right],\quad\mathbf{P}\textrm{-a.s., }t\geq0,\label{eq:statementMK1}
\end{equation}
for all $h\geq0$ and for every measurable function $f$ which is nonnegative or such that $f(\beta_{t+h})$
is integrable.
\end{thm}
\begin{proof}
For $t=0$ the statement is clear. Let us assume $t>0$. On the set
$\{ \tau\leq t\}$ we have $\mathbf{P}$-a.s.
\begin{equation*}
\mathbf{E}\left[f\left(\beta_{t+h}\right)|\mathcal{F}_{t}^{0}\right]\mathbb{I}_{\left\{ \tau\leq t\right\} } =\mathbf{E}\left[f\left(0\right)\mathbb{I}_{\left\{ \tau\leq t\right\} }|\mathcal{F}_{t}^{0}\right]
  =f\left(0\right)\mathbb{I}_{\left\{ \tau\leq t\right\}}
  =f\left(0\right)\mathbb{I}_{\left\{ \beta_{t}=0\right\}}\,, 
\end{equation*}
which is a measurable function with respect to $\sigma(\beta_{t})$ and hence (\ref{eq:statementMK1}) is valid on $\{ \tau\leq t\}$ $\mathbf{P}$-a.s. Now we have to prove 
\[
\mathbf{E}\left[f\left(\beta_{t+h}\right)|\mathcal{F}_{t}^{0}\right]\mathbb{I}_{\left\{ \tau>t\right\} }=\mathbf{E}\left[f\left(\beta_{t+h}\right)|\beta_{t}\right]\mathbb{I}_{\left\{ \tau>t\right\} },\;\mathbf{P}\textrm{-a.s.}
\]
Both sides being $\mathcal{F}_{t}^{P}$-measurable, it suffices to
verify that for all $A\in\mathcal{F}_{t}^{0}$ we have 
\begin{equation}
\int_{A\cap\left\{ \tau>t\right\} }f\left(\beta_{t+h}\right)d\mathbf{P}=\int_{A\cap\left\{ \tau>t\right\} }\mathbf{E}\left[f\left(\beta_{t+h}\right)|\beta_{t}\right]d\mathbf{P}.\label{eq:equivalentMK1}
\end{equation}
We observe that $\mathcal{F}_{t}^{0}$ is generated by 
\[
\beta_{t_{n}},\ \xi_n:=\frac{\beta_{t_{n}}}{t_{n}}-\frac{\beta_{t_{n-1}}}{t_{n-1}},\ \xi_{n-1}:=\frac{\beta_{t_{n-1}}}{t_{n-1}}-\frac{\beta_{t_{n-2}}}{t_{n-2}},\ ..., \ \xi_1:=\frac{\beta_{t_{1}}}{t_{1}}-\frac{\beta_{t_{0}}}{t_{0}}\,, 
\]
$0<t_{0}<t_{1}<...<t_{n}=t$, for $n$ running through $\mathbb{N}$. By the monotone class theorem (see, e.g., \cite[I.19, I.21]{key-7}) it is sufficient to prove (\ref{eq:equivalentMK1})
for sets $A$ of the form 
$A=\{ \beta_{t}\in B,\,\xi_{1}\in B_{1},\ldots,\,\xi_{n}\in B_{n}\}$
with $B, B_{1}, B_{2},\ldots, B_{n}\in \mathcal{B}(\mathbb{R}),\, n\geq1$. Let $g:=\mathbb{I}_{B}$ and $L:=\mathbb{I}_{B_{1}\times B_{2}\times\cdots\times B_{n}}$. Then we have the equality $\mathbb{I}_{A}\!=\!g(\beta_{t})L(\xi_{1},\ldots,\,\xi_{n})$ and, setting 
\[
\eta_{k}:=\frac{W_{t_{k}}}{t_{k}}-\frac{W_{t_{k-1}}}{t_{k-1}},\quad k=1,\ldots, n,
\]
we have $\xi_{k}=\eta_{k}$ on $\{ \tau>t\}$, $k=1, \ldots, n$. But, for $r>t$, the random vector $(\eta_{1},\ldots, \eta_{n},\,\beta_{t}^{r},\beta_{t+h}^{r})$
is Gaussian and, denoting by $\textrm{cov}(X,Y)$ the covariance between two random variables $X$ and $Y$, we have that 
$\textrm{cov}(\eta_{k},\beta_{t}^{r})=\textrm{cov}(\eta_{k},\beta_{t+h}^{r})=0,\; k=1, \ldots, n$.
Thus $(\eta_{1},...,\eta_{n})$ is independent of $(\beta_{t}^{r},\beta_{t+h}^{r})$
and, with the notation $H(x,y):=f(x)g(y)$,
we also have that $L(\eta_{1},...,\eta_{n})$ is independent
of $H(\beta_{t+h}^{r},\beta_{t}^{r})$. Now we can state
the following lemma which will allow to complete the proof of Theorem
\ref{TEO:The-information-process}.
\begin{lem}
\label{lem:The-random-variables} Let $H:\, \mathbb{R}^2\mapsto\mathbb{R}$ be a measurable function. Suppose that $H$ is nonnegative or such that $\mathbf{E}\left[\left| H\left(\beta_{t+h},\beta_{t}\right)\right|\right]<+\infty$. Then the random variables $H(\beta_{t+h},\beta_{t})\mathbb{I}_{\{ \tau>t\} }$
and $L(\eta_{1},...,\eta_{n})$ are uncorrelated:
\begin{align*}
\mathbf{E}\left[H\left(\beta_{t+h},\beta_{t}\right)\mathbb{I}_{\left\{ \tau>t\right\}}L\left(\eta_{1},...,\eta_{n}\right)\right] &=\mathbf{E}\left[H\left(\beta_{t+h},\beta_{t}\right)\mathbb{I}_{\left\{ \tau>t\right\} }\right]\mathbf{E}\left[L\left(\eta_{1},...,\eta_{n}\right)\right]
\end{align*}
\end{lem}
\begin{proof}
Using the formula of total probability and Lemma \ref{lem:If--is} we get
\begin{align*}
\lefteqn{\mathbf{E}\left[H\left(\beta_{t+h},\beta_{t}\right)\mathbb{I}_{\left\{ \tau>t\right\} }L\left(\eta_{1},...,\eta_{n}\right)\right]}\\
&=\int_{\left(t,+\infty\right)}\mathbf{E}\left[H\left(\beta_{t+h},\beta_{t}\right)\mathbb{I}_{\left\{ \tau>t\right\} }L\left(\eta_{1},...,\eta_{n}\right)|\tau=r\right]dF\left(r\right)\\
&=\int_{\left(t,+\infty\right)}\mathbf{E}\left[H\left(\beta_{t+h}^{r},\beta_{t}^{r}\right)L\left(\eta_{1},...,\eta_{n}\right)\right]dF\left(r\right)\\
&=\int_{\left(t,+\infty\right)}\mathbf{E}\left[H\left(\beta_{t+h}^{r},\beta_{t}^{r}\right)\right]dF\left(r\right)\mathbf{E}\left[L\left(\eta_{1},...,\eta_{n}\right)\right]\\
&=\mathbf{E}\left[H\left(\beta_{t+h},\beta_{t}\right)\mathbb{I}_{\left\{ \tau>t\right\} }\right]\mathbf{E}\left[L\left(\eta_{1},...,\eta_{n}\right)\right].
\end{align*}
The proof of the lemma is finished.
\end{proof}
We now prove (\ref{eq:equivalentMK1}) for our special choice of $A$.
From Lemma \ref{lem:The-random-variables} above we have
\begin{align*}
\int_{A\cap\left\{ \tau>t\right\} }f\left(\beta_{t+h}\right)d\mathbf{P} & =\mathbf{E}\left[H\left(\beta_{t+h},\beta_{t}\right)\mathbb{I}_{\left\{ \tau>t\right\} }L\left(\xi_{1},...,\xi_{n}\right)\right]\\
 & =\mathbf{E}\left[H\left(\beta_{t+h},\beta_{t}\right)\mathbb{I}_{\left\{ \tau>t\right\} }L\left(\eta_{1},...,\eta_{n}\right)\right]\\
 &  =\mathbf{E}\left[f\left(\beta_{t+h}\right)g\left(\beta_{t}\right)\mathbb{I}_{\left\{ \tau>t\right\} }\right]\mathbf{E}\left[L\left(\eta_{1},...,\eta_{n}\right)\right]\\
 & =\mathbf{E}\left[\mathbf{E}\left[f\left(\beta_{t+h}\right)|\beta_{t}\right]g\left(\beta_{t}\right)\mathbb{I}_{\left\{ \tau>t\right\} }\right]\mathbf{E}\left[L\left(\eta_{1},...,\eta_{n}\right)\right]\\
 & =\mathbf{E}\left[\mathbf{E}\left[f\left(\beta_{t+h}\right)|\beta_{t}\right]g\left(\beta_{t}\right)\mathbb{I}_{\left\{ \tau>t\right\} }L\left(\eta_{1},...,\eta_{n}\right)\right]\\
 & =\mathbf{E}\left[\mathbf{E}\left[f\left(\beta_{t+h}\right)|\beta_{t}\right]g\left(\beta_{t}\right)\mathbb{I}_{\left\{ \tau>t\right\} }L\left(\xi_{1},...,\xi_{n}\right)\right]\\
 & =\int_{A\cap\left\{ \tau>t\right\} }\mathbf{E}\left[f\left(\beta_{t+h}\right)|\beta_{t}\right]d\mathbf{P},
\end{align*}
which proves that (\ref{eq:equivalentMK1}) is true and this ends
the proof.
\end{proof}
\begin{rem} Note that the Markov property is trivially extended to the completed filtration $\mathbb{F}^{P}$.
\end{rem}
\section{\label{sec:Conditional-Expectations} Bayes Estimates of the Default Time $\tau$}
\noindent The basic aim of the present section is to provide estimates of the a priori unknown default time $\tau$ based on the observation of the information process $\beta$ up to time $t$. For fixed $t\geq 0$, the observation is represented by the $\sigma$-algebra $\mathcal{F}^P_t$ and, because of the Markov property, the observation of $\beta_t$ would be sufficient. To this end it is natural to exploit the Bayesian approach.  

The idea is to use the knowledge gained from the observation of the flow $(\beta_t, \ t\geq 0)$ for updating the initial knowledge on $\tau$. At time 0, the market agents have only  \textit{a priori} knowledge about $\tau$, represented by its distribution function
$F$. As time is increasing, information concerning the default becomes available. Using the Bayes theorem (recalled in the Appendix for easy reference), the  \textit{a posteriori} distribution of $\tau$ based on the observation of $\beta$ up to $t$ can be derived and in this way agents can update their initial knowledge obtaining a sharper estimate of the default time $\tau$.  

In this section the $\sigma$-algebra generated by the future of $\beta$
at time $t$ is denoted by $\mathcal{F}^P_{t,\infty}:=\sigma(\beta_{s},\, t\leq s\leq+\infty)\vee\mathcal{N}_{P}$.
The following is a standard result on Markov processes:
\begin{lem}
\label{lem:basic MKV}Let $X=(X_{t},\, t\geq0)$ be a stochastic
process adapted to a filtration $\mathbb{F}=(\mathcal{F}_{t})_{t\geq0}$.
Then the following are equivalent:
\begin{enumerate}
\item[(i)] $X$ is Markov with respect to $\mathbb{F}$.
\item[(ii)] For each $t\geq0$ and bounded (or nonnegative) $\sigma(X_{s},\, s\geq t)$-measurable random variable $Y$ one has
\[
\mathbf{E}\left[Y|\mathcal{F}_{t}\right]=\mathbf{E}\left[Y|X_{t}\right], \; \mathbf{P}\textrm{-a.s.}
\]
\end{enumerate}
\end{lem}
\begin{proof}
See \cite[Ch. I, Theorem (1.3)]{key-5}.
\end{proof}
The next proposition describes the structure of the a posteriori distribution of $\tau$ based on the observation of $\mathcal{F}^P_t$.
\begin{prop}
\label{prop:HJ1} For all $t,u\geq0$, it holds
\begin{equation}\label{prop:HJ1second}
\mathbf{P}\left(\tau\leq u|\mathcal{F}_{t}^{P}\right)=\mathbb{I}_{\left\{ \tau\leq t\wedge u\right\} }+\mathbf{P}\left(t<\tau\leq u|\beta_{t}\right)\mathbb{I}_{\left\{ t<\tau\right\} },\;\mathbf{P}\textrm{-a.s.}
\end{equation}
\end{prop}
\begin{proof}
Obviously, we have $\{ \tau\leq u\} =\{ \tau\leq t\wedge u\} \cup\{ t<\tau\leq u\}$. The first set of the right-hand side of the above decomposition yields the first summand of the statement. Then it suffices to observe that we have the relation
$\{ t<\tau\leq u\} =\{ \beta_{t}\neq0,\beta_{u}=0\}$      $\mathbf{P}\textrm{-a.s.}$ where the set on the right-hand side of the above equality belongs to $\mathcal{F}^P_{t,\infty}$. It remains to apply Lemma \ref{lem:basic MKV} in order to complete the proof of the statement.
\end{proof}
Recalling that $F$ denotes the distribution function of $\tau$ and
formula (\ref{eq:bbdensity}) for the definition of the function $\varphi_{t}(r,x)$ (which is equal to the density of the Brownian bridge $\beta_{t}^{r}$ at time $t<r$), we have the following result which provides the explicit form of the a posteriori distribution of $\tau$ based on the observation of $\beta$ up to $t$.
\begin{thm}
\label{lem:PRE-theo-4.3}Let $t>0$. Then, for all $u>0$, $\mathbf{P}$-a.s.
\begin{align}
\mathbf{P}\left(\tau\leq u|\mathcal{F}_{t}^{P}\right) =\mathbb{I}_{\left\{\tau\leq t\wedge u\right\}}+\frac{{\displaystyle \int_{\left(t,u\right]}}\varphi_{t}\left(r,\beta_{t}\right)dF(r)}{{\displaystyle \int_{\left(t,+\infty\right)}}\varphi_{t}\left(v,\beta_{t}\right)dF(v)}\,\mathbb{I}_{\left\{ t<\tau\right\}}\,.\label{eq:condexptau-1}
\end{align}
\end{thm}
\begin{proof}
The result is a consequence of Proposition \ref{prop:HJ1} and
the Bayes formula (see Corollary \ref{cor:(Bayes-formula)-a.s.}).
\end{proof}
Theorem \ref{lem:PRE-theo-4.3} can be extended to functions $g$ on $\mathbb{R}_+$ as it will be stated in the following corollary.
\begin{cor}
\label{LEM:prop_cond_exp_tau}Let $t>0,\; g:\mathbb{R}_{+}\rightarrow\mathbb{R}$
be a Borel function such that $\mathbf{E}\left[|g(\tau)|\right]<+\infty$.
Then, $\mathbf{P}$-a.s.,
\begin{align}
\mathbf{E}\left[g\left(\tau\right)|\mathcal{F}_{t}^{P}\right] & =g\left(\tau\right)\mathbb{I}_{\left\{ \tau\leq t\right\} }+\frac{{\displaystyle \int_{\left(t,+\infty\right)}}g\left(r\right)\varphi_{t}\left(r,\beta_{t}\right)dF(r)}{{\displaystyle \int_{\left(t,+\infty\right)}}\varphi_{t}\left(v,\beta_{t}\right)dF(v)}\,\mathbb{I}_{\left\{ t<\tau\right\}}\,. \label{eq:condexptau2}
\end{align}
\end{cor}
\begin{proof}
If the function $g$ is bounded then the statement immediately follows
by an application of the monotone class theorem to simple functions where it is possible to use Theorem \ref{lem:PRE-theo-4.3}. In the general case $g$ has to be approximated pointwise by bounded functions and by passing to the limit.
\end{proof}
\begin{rem}
We point out that the function $\phi_{t}$ defined by
\begin{equation}
\phi_{t}\left(r,x\right):=\frac{\varphi_{t}\left(r,x\right)}{{\displaystyle \int_{\left(t,+\infty\right)}\varphi_{t}\left(v,x\right)dF\left(v\right)}}, \; \left(r,t\right)\in\left(0,+\infty\right)\times\mathbb{R}_+,\; x\in\mathbb{R}\,, \label{eq:densitybeta2}
\end{equation}
is, for $t<r$, the a posteriori density function of $\tau$ on $\{\tau>t\}$ based on the observation $\beta_t=x$ (see Corollary \ref{cor:The-conditional-density}).
Then relation (\ref{eq:condexptau-1}),
representing the a posteriori distribution of $\tau$ based on the observation of $\mathcal{F}_t^P$, can be
rewritten as 
\[
\mathbf{P}\left(\tau\leq u|\mathcal{F}_{t}^{P}\right)=\mathbb{I}_{\left\{ \tau\leq t\right\}} +{\displaystyle \int_{\left(t,u\right]}}\phi_{t}\left(r,\beta_{t}\right)dF\left(r\right)\mathbb{I}_{\left\{ t<\tau\right\} },\;\mathbf{P}\textrm{-a.s.},
\]
while (\ref{eq:condexptau2}) is equal to the expression 
\[
\mathbf{E}\left[g\left(\tau\right)|\mathcal{F}_{t}^{P}\right]=g\left(\tau\right)\mathbb{I}_{\left\{ \tau\leq t\right\} }+{\displaystyle \int_{\left(t,+\infty\right)}g\left(r\right)}\,\phi_{t}\left(r,\beta_{t}\right)dF\left(r\right)\mathbb{I}_{\left\{ t<\tau\right\} },\;\mathbf{P}\textrm{-a.s.}
\]
Here it is possible to see how the Bayesian estimate of $\tau$ given above provides a better knowledge on the default time $\tau$ through the observation of the information process $\beta$ at time $t$.
\end{rem}
\section{Extensions of the Bayes Estimates}
\noindent In this section we shall deal with an extension of the Bayes estimates of $\tau$ provided in Section \ref{sec:Conditional-Expectations}. Roughly speaking we shall derive formulas which include the Bayes estimates discussed in Section \ref{sec:Conditional-Expectations} as well as the prognose of the information process 
$\beta$ at some time $u$, the latter being related with the Markov property which has been proven in Section \ref{sec:Definition-and-First} (see Theorem \ref{TEO:The-information-process}).
First we will state a lemma that will be used in the proof of the main results of this section.
\begin{lem}
\label{lem:AUXlemma}Let $0\leq t<u$ and $g$ be a measurable function
on $(0,+\infty)\times\mathbb{R}$ such that $g(\tau,\beta_{u})$ is integrable. Then it holds $\mathbf{P}$-a.s.
\begin{align*}
\mathbf{E}\left[g\left(\tau,\beta_{u}\right)\mathbb{I}_{\left\{ t<\tau\right\} }|\mathcal{F}_{t}^{P}\right] & =\mathbf{E}\left[g\left(\tau,\beta_{u}\right)|\beta_{t}\right]\mathbb{I}_{\left\{ t<\tau\right\} }\\
 & =\mathbf{E}\left[\mathbf{E}\left[g\left(\tau,\beta_{u}\right)|\sigma\left(\tau\right)\vee\sigma\left(\beta_{t}\right)\right]|\beta_{t}\right]\mathbb{I}_{\left\{ t<\tau\right\} }\\
 & =\mathbf{E}\left[\left(\mathbf{E}\left[g\left(r,\beta_{u}^{r}\right)|\beta_{t}^{r}\right]\right)_{r=\tau}|\beta_{t}\right]\mathbb{I}_{\left\{ t<\tau\right\} }.
\end{align*}
\end{lem}
\begin{proof}
It is clear that the first equality holds true due to the fact that
$g(\tau,\beta_{u})\mathbb{I}_{\{ t<\tau\} }$
is $\mathcal{F}_{t,\infty}$-measurable and, hence, Lemma \ref{lem:basic MKV}
can be applied. The second equality is obvious. Let $h$ be an arbitrary bounded Borel function. Using
Lemma \ref{lem:If--is} we have that
\begin{align*}
\mathbf{E}\left[g\left(\tau,\beta_{u}\right)h\left(\beta_{t}\right)\mathbb{I}_{\left\{ t<\tau\right\} }\right] & =\int_{\left(t,+\infty\right)}\mathbf{E}\left[g\left(r,\beta_{u}^{r}\right)h\left(\beta_{t}^{r}\right)\right]dF\left(r\right)\\
 & =\int_{\left(t,+\infty\right)}\mathbf{E}\left[\mathbf{E}\left[g\left(r,\beta_{u}^{r}\right)h\left(\beta_{t}^{r}\right)|\beta_{t}^{r}\right]\right]dF\left(r\right)\\
 & =\int_{\left(t,+\infty\right)}\mathbf{E}\left[\mathbf{E}\left[g\left(r,\beta_{u}^{r}\right)|\beta_{t}^{r}\right]h\left(\beta_{t}^{r}\right)\right]dF\left(r\right)\\
 & =\mathbf{E}\left[\mathbf{E}\left[g\left(r,\beta_{u}^{r}\right)|\beta_{t}^{r}\right]_{r=\tau}h\left(\beta_{t}\right)\mathbb{I}_{\left\{ t<\tau\right\} }\right]\,,
\end{align*}
that is,
\[
\mathbf{E}\left[\left(g\left(\tau,\beta_{u}\right)-\mathbf{E}\left[g\left(r,\beta_{u}^{r}\right)|\beta_{t}^{r}\right]_{r=\tau}\right)h\left(\beta_{t}\right)\mathbb{I}_{\left\{ t<\tau\right\} }\right]=0.
\]
But $h$ is arbitrary  and thus
\begin{align*}
\mathbf{E}\left[g\left(\tau,\beta_{u}\right)|\beta_{t}\right] & =\mathbf{E}\left[\mathbf{E}\left[g\left(r,\beta_{u}^{r}\right)|\beta_{t}^{r}\right]_{r=\tau}|\beta_{t}\right],
\end{align*}
$\mathbf{P}$-a.s. on $\{ t<\tau\} $, for $t<u$.
\end{proof}
The next theorem is prepared by the following result. 
\begin{prop}
\label{cor:Let--be-BIS} Let $t\geq0$ and $g$ be a measurable function
such that $g(\tau,\beta_{t})$ is integrable. Then, $\mathbf{P}$-a.s.,
\[
\mathbf{E}\left[g\left(\tau,\beta_{t}\right)|\mathcal{F}_{t}^{\beta}\right]=g\left(\tau,0\right)\mathbb{I}_{\left\{ \tau\leq t\right\} }+\int_{\left(t,+\infty\right)}g\left(r,\beta_{t}\right)\phi_{t}\left(r,\beta_{t}\right)dF\left(r\right)\mathbb{I}_{\left\{ t<\tau\right\} }\,.
\]
\end{prop}
\begin{proof}
The statement is clear on the set $\{ \tau\leq t\} $.
On the set $\{ t<\tau\}$, we first prove it for bounded measurable functions $g$ and, by a monotone class argument, it suffices to consider functions $g$ of the form $g(\tau,\beta_{t})=g_{1}(\tau)g_{2}(\beta_{t})$ where $g_1$ and $g_2$ are bounded measurable functions. Then we have that
\begin{align*}
\mathbf{E}\left[g\left(\tau,\beta_{t}\right)|\mathcal{F}_{t}^{\beta}\right]\mathbb{I}_{\left\{ t<\tau\right\} } & =\frac{\int_{\left(t,+\infty\right)}g_{1}\left(r\right)\varphi_{t}\left(r,\beta_{t}\right)dF(r)}{\int_{\left(t,+\infty\right)}\varphi_{t}\left(v,\beta_{t}\right)dF(v)}\mathbb{I}_{\left\{ t<\tau\right\} }\,g_{2}\left(\beta_{t}\right)\\
 & =\int_{\left(t,+\infty\right)}g\left(r,\beta_{t}\right)\phi_{t}\left(r,\beta_{t}\right)dF\left(r\right)\mathbb{I}_{\left\{ t<\tau\right\} }.
\end{align*}
If $g$ is a nonnegative measurable function, we can apply the above result to the functions $g_N:=g\wedge N$
and  using the monotone convergence theorem we obtain the asserted equality for $g$. Finally, in the general case, the result is true for the positive and negative parts $g^+$ and $g^-$ of $g$ and for $g=g^+-g^-$ the equality follows from the linearity of both sides. 
\end{proof}
We can now state the main result of this section.
\begin{thm}
\label{thm:Let--and} Let $0<t<u$ and $g$ be a measurable function
defined on $(0,+\infty)\times\mathbb{R}$ such that $\mathbf{E}\left[|g(\tau,\,\beta_{u})|\right]<+\infty$.
Then, $\mathbf{P}$-a.s.
\begin{align*}
\lefteqn{\mathbf{E}\left[g\left(\tau,\,\beta_{u}\right)|\mathcal{F}_{t}^{P}\right]=\mathbf{E}\left[g\left(\tau,\,\beta_{u}\right)|\beta_t\right]}\\
=&\;g\left(\tau,0\right)\mathbb{I}_{\left\{ \tau\leq t\right\} }+\int_{\left(t,u\right]}g\left(r,0\right)\phi_{t}\left(r,\beta_{t}\right)dF\left(r\right)\,\mathbb{I}_{\left\{ t<\tau\right\}}\\
&+\!\!\int\limits_{\left(u,+\infty\right)}\!\!\int\limits_{\mathbb{R}}g\left(r,y\right)p\big(\frac{r-u}{r-t}\left(u-t\right),\, y,\,\frac{r-u}{r-t}\beta_{t}\big)dy\phi_{t}\left(r,\beta_{t}\right)dF\left(r\right)\mathbb{I}_{\left\{ t<\tau\right\} }\,, \label{eq:prova1}
\end{align*}
where $p(t,\cdot,y)$ is the Gaussian density with mean $y$
and variance $t$.
\end{thm}
\begin{proof}
On the set $\{ \tau\leq t\} $ the statement is a consequence
of the fact that $\tau$ is an $\mathbb{F}^{P}$-stopping time and
that $\beta_{u}=0$ on $\{\tau\leq t<u\}$. On the set $\{ t<\tau\}$, from Lemma \ref{lem:AUXlemma} we have
\begin{align*}
\mathbf{E}\left[g\left(\tau,\,\beta_{u}\right)|\mathcal{F}_{t}^{P}\right]\mathbb{I}_{\left\{ t<\tau\right\} } & =\mathbf{E}\left[g\left(\tau,\,\beta_{u}\right)|\beta_{t}\right]\mathbb{I}_{\left\{ t<\tau\right\} }\\
 & =\mathbf{E}\left[g\left(\tau,0\right)\mathbb{I}_{\left\{ t<\tau\leq u\right\} }|\beta_{t}\right]+\mathbf{E}\left[g(\tau,\beta_{u})\mathbb{I}_{\left\{ u<\tau\right\} }|\beta_{t}\right]\,,
\end{align*}
$\mathbf{P}$-a.s. We remark that due to Corollary \ref{LEM:prop_cond_exp_tau}
\[
\mathbf{E}\left[g\left(\tau,0\right)\mathbb{I}_{\left\{ t<\tau\leq u\right\} }|\beta_{t}\right]=\int_{\left(t,u\right]}g\left(r,0\right)\phi_{t}\left(r,\beta_{t}\right)dF\left(r\right)\,\mathbb{I}_{\{t<\tau\}}\,.
\]
On the other hand, from (\ref{eq:cond exp ext bb}), for $t<u<r$,
\begin{align}
\mathbf{E}\left[g\left(r,\beta_{u}^{r}\right)|\beta_{t}^{r}=x\right]&=\int_{\mathbb{R}}g\left(r,y\right)p\left(\frac{r-u}{r-t}\left(u-t\right),\,y,\,\frac{r-u}{r-t}x\right)dy\nonumber\\
&=:G_{t,u}\left(r,x\right)\,. \label{eq:prova3}
\end{align}
It follows from Lemma \ref{lem:AUXlemma}
and from (\ref{eq:prova3}) that, on $\{ t<\tau\}$  $\mathbf{P}$-a.s.,
\begin{align*}
\mathbf{E}\left[g(\tau,\beta_{u})\mathbb{I}_{\left\{ u<\tau\right\} }|\beta_{t}\right] & =\mathbf{E}\left[\left(\mathbf{E}\left[g\left(r,\beta_{u}^{r}\right)\mathbb{I}_{\left\{ u<r\right\} }|\beta_{t}^{r}\right]\right)_{r=\tau}|\beta_{t}\right]\\
 & =\mathbf{E}\left[\left(G_{t,u}\left(r,\beta_{t}^{r}\right)\right)_{r=\tau}|\beta_{t}\right]\\
 & =\mathbf{E}\left[G_{t,u}\left(\tau,\beta_{t}\right)|\beta_{t}\right]\\
 & =\int_{\left(u,+\infty\right)}G_{t,u}\left(r,\beta_{t}\right)\phi_{t}\left(r,\beta_{t}\right)dF\left(r\right),
\end{align*}
where the latter equality follows from Proposition \ref{cor:Let--be-BIS}.
\end{proof}
\begin{ex}
As an immediate consequence of Theorem \ref{thm:Let--and}, for $t<u$, we can calculate the conditional expectation of $\beta_{u}$ given $\beta_{t}$ by
\[
\mathbf{E}\left[\beta_{u}|\mathcal{F}_{t}^{P}\right]=\beta_{t}\int_{\left(u,+\infty\right)}\frac{r-u}{r-t}\,\phi_{t}\left(r,\beta_{t}\right)dF\left(r\right)\mathbb{I}_{\left\{ t<\tau\right\} },\; \mathbf{P}\textrm{-a.s.},
\]
and the conditional distribution of $\beta_u$ given $\beta_t$: For $\Gamma\in\mathcal{B}(\mathbb{R})$, $\mathbf{P}$-a.s.,  
\begin{multline}
\mathbf{P}\left(\beta_{u}\in\Gamma|\mathcal{F}_{t}^{P}\right)=\mathbb{I}_{\left\{ 0\in\Gamma\right\}}\,\mathbb{I}_{\left\{ \tau\leq t\right\} } +\mathbb{I}_{\left\{ 0\in\Gamma\right\} } \int_{\left(t,u\right]}\phi_{t}\left(r,\beta_{t}\right)dF\left(r\right)\,\mathbb{I}_{\left\{ \tau\leq t\right\} }\\
+\int_{\left(u,+\infty\right)}\int_{\Gamma}p\big(\frac{r-u}{r-t}\left(u-t\right),\, y,\,\frac{r-u}{r-t}\beta_{t}\big)dy\,\phi_{t}\left(r,\beta_{t}\right)dF\left(r\right)\mathbb{I}_{\left\{ t<\tau\right\} }.\label{eq:prova2}
\end{multline}
\end{ex}
\begin{rem}
\label{rem:Strong_Inhom_MK} From the factor $\left(r-u\right)/\left(r-t\right)$
in (\ref{eq:prova2}) we see that the process $\beta$ cannot
be a homogeneous Markov process because $\mathbf{P}(\beta_{u}\in\Gamma|\mathcal{F}_{t}^{P})$
does not depend only on $u-t$ and $(\beta_{t},\Gamma)$. 
\end{rem}
\section{Markov Property\label{sec:The-Markov-property2}}
\noindent In this section we are going to strengthen Theorem \ref{TEO:The-information-process} on the Markov property of the information process $\beta$. We shall prove that $\beta$ is not only a Markov process with respect to the filtration $\mathbb{F}^0$ (or $\mathbb{F}^P$) but also with respect to $\mathbb{F}^\beta$, the smallest filtration containing  $\mathbb{F}^0$ and satisfying the usual conditions. As an important consequence, it turns out that the filtrations $\mathbb{F}^P$ and $\mathbb{F}^\beta$ are equal which amounts to saying that the filtration $\mathbb{F}^P$ is right-continuous. The result is stated in the following theorem.
\begin{thm}
\label{thm:The-process--1}The process $\beta$ is a Markov process
with respect to the filtration $\mathbb{F}^{\beta}$, i.e.,
\[
\mathbf{E}\left[g\left(\beta_{u}\right)|\mathcal{F}_{t}^{\beta}\right]=\mathbf{E}\left[g\left(\beta_{u}\right)|\beta_{t}\right],\;\mathbf{P}\textrm{-a.s.}
\]
for all $0\leq t<u$ and all measurable functions $g$ such that $g(\beta_{u})$ is integrable.
\end{thm}
\begin{proof}
The proof is divided into two main parts. In the first one we prove
the statement of the above theorem for $t>0$, while in the second
part we consider the case $t=0$. Throughout the proof we can assume
without loss of generality that the function $g$ is continuous and
bounded by some constant $M\in\mathbb{R}_+$. For the first part of the proof, let $t>0$ be a strictly positive
real number and let $(t_{n})_{n\in\mathbb{N}}$ be a decreasing
sequence converging to $t$ from above: $0<t<...<t_{n+1}<t_{n}<u,\: t_{n}\downarrow t$ as $n\rightarrow\infty$.
From the definition of $\mathcal{F}_{t}^{\beta}$ we have that 
$\mathcal{F}_{t}^{\beta}=\bigcap_{n\in\mathbb{N}}\mathcal{F}_{t_{n}}^{P}$
where we recall that $\mathcal{F}_{v}^{P}=\sigma(\beta_{s},\,0\leq s\leq v)\vee\mathcal{N}_{P}$, $v\geq 0$.
Consequently,
\[
\mathbf{E}\left[g\left(\beta_{u}\right)|\mathcal{F}_{t}^{\beta}\right]=\lim_{n\rightarrow\infty}\mathbf{E}\left[g\left(\beta_{u}\right)|\mathcal{F}_{t_{n}}^{P}\right]\,, \; \mathbf{P}\textrm{-a.s.}
\]
From Theorem \ref{thm:Let--and} and the definition of $G_{t_n,u}$ by \eqref{eq:prova3} we know that, as $t_{n}<u$, $\mathbf{P}$-a.s.
\begin{eqnarray}
\nonumber\lefteqn{\mathbf{E}\left[g\left(\beta_{u}\right)|\mathcal{F}_{t_{n}}^{P}\right]=\mathbf{E}\left[g\left(\beta_{u}\right)|\beta_{t_{n}}\right]}\\
&=&g\left(0\right)\mathbb{I}_{\left\{ \tau\leq t_{n}\right\} }+g\left(0\right)\int_{\left(t_{n},u\right]}\phi_{t_{n}}\left(r,\beta_{t_{n}}\right)dF\left(r\right)\,\mathbb{I}_{\left\{ t_{n}<\tau\right\} }\label{eq:equiv1}\\
\nonumber&&+\int_{\left(u,+\infty\right)}G_{t_n,u}\left(r,\beta_{t_n}\right)\,\phi_{t_{n}}\left(r,\beta_{t_{n}}\right)dF\left(r\right)\,\mathbb{I}_{\left\{ t_{n}<\tau\right\} }\,.
\end{eqnarray}
We want to prove that
\[
\lim_{n\rightarrow\infty}\mathbf{E}\left[g\left(\beta_{u}\right)|\mathcal{F}_{t_{n}}^{P}\right]=\mathbf{E}\left[g\left(\beta_{u}\right)|\beta_{t}\right],\;\mathbf{P}\textrm{-a.s.}
\]
Using (\ref{eq:equiv1}) and Theorem \ref{thm:Let--and} we see
that this latter relation holds true if the following two identities are satisfied, $\mathbf{P}$-a.s. on $\{t<\tau\}$:
\begin{align}
\lim_{n\rightarrow\infty}\int_{\left(t_{n},u\right]}\phi_{t_{n}}\left(r,\beta_{t_{n}}\right)dF\left(r\right)&= \int_{\left(t,u\right]}\phi_{t}\left(r,\beta_{t}\right)dF\left(r\right)\,,\label{eq:star2}\\
\lim_{n\rightarrow\infty}\int\limits_{\left(u,+\infty\right)}\!\!\!G_{t_n,u}\left(r,\beta_{t_{n}}\right)\phi_{t_{n}}\left(r,\beta_{t_{n}}\right)dF\left(r\right)&=\!\!\! \int\limits_{\left(u,+\infty\right)}\!\!\!G_{t,u}\left(r,\beta_{t}\right)\phi_{t}\left(r,\beta_{t}\right)dF\left(r\right)\,.\label{eq:star3}
\end{align}
Relation (\ref{eq:star2}) can be derived as follows.

\vspace{6pt}
\noindent\textit{Proof of (\ref{eq:star2}).}
Recalling that by (\ref{eq:densitybeta2})
\[
\phi_{t_{n}}\left(r,\beta_{t_{n}}\right)=\frac{\varphi_{t_{n}}\left(r,\beta_{t_{n}}\right)}{{\displaystyle \int_{\left(t_{n},+\infty\right)}}\varphi_{t_{n}}\left(v,\beta_{t_{n}}\right)dF(v)},\quad t_n<r\,,
\]
the integral on the left-hand side
of (\ref{eq:star2}) can be rewritten as 
\[
\int_{\left(t_{n},u\right]}\phi_{t_{n}}\left(r,\beta_{t_{n}}\right)dF\left(r\right)=\frac{{\displaystyle \int_{\left(t_{n},u\right]}}\varphi_{t_{n}}\left(r,\beta_{t_{n}}\right)dF(r)}{{\displaystyle \int_{\left(t_{n},+\infty\right)}}\varphi_{t_{n}}\left(v,\beta_{t_{n}}\right)dF(v)}\,.
\]
The plan is to apply Lebesgue's bounded convergence theorem to the numerator and the denominator of the above expression. To this end we have to prove $\mathbf{P}$-a.s. pointwise convergence and uniform boundedness of the integrand $\varphi_{t_n}(\cdot, \beta_{t_n})$  $\mathbf{P}$-a.s. We begin by focusing our attention on the function
$(t,r,x)\mapsto\varphi_{t}(r,x)$ defined by \eqref{eq:bbdensity}, which is continuous on $(0,+\infty)\times[0,+\infty)\times\mathbb{R}\backslash\{ 0\}$.
Setting $\varphi_{t}(+\infty,x):=p(t,x,0)$
for every $t>0$ and $x\in\mathbb{R},$ we see that the resulting
function $(t,r,x)\mapsto\varphi_{t}(r,x)$, now defined on 
$(0,+\infty)\times[0,+\infty]\times\mathbb{R}\backslash\{ 0\}$,
is continuous, too. Hence $\lim_{n\rightarrow \infty}\varphi_{t_{n}}(r,\beta_{t_{n}})=\varphi_{t}(r,\beta_{t})$, $\mathbf{P}$-a.s. on $\tau>t$, providing pointwise convergence. For this we note that $\beta_{t}\neq0$ $\mathbf{P}$-a.s. on the set $\{t<\tau\}$. Now we fix $\omega\in\Omega$ such that $t<\tau(\omega)$ and $\beta_{t}(\omega)\neq0$.
Then the set $\{ t_{n}:\, n\in\mathbb{N}\} \times(t,+\infty]\times\{ \beta_{t_{n}}(\omega):\, n\in\mathbb{N}\} $
is obviously contained in a compact subset of $(0,+\infty)\times[0,+\infty]\times\mathbb{R}\backslash\{0\}$
(depending on $\omega$). This implies that $\varphi_{t_{n}}(r,\beta_{t_{n}}(\omega))$
is bounded (by a constant depending on $\omega$). Using Lebesgue's bounded convergence theorem we can conclude 
\[
\lim_{n\rightarrow\infty}\int_{\left(t_{n},u\right]}\varphi_{t_{n}}\left(r,\beta_{t_{n}}\left(\omega\right)\right)dF\left(r\right)=\int_{\left(t,u\right]}\varphi_{t}\left(r,\beta_{t}\left(\omega\right)\right)dF\left(r\right).
\]
Consequently, we have proven that 
\[
\lim_{n\rightarrow\infty}\int_{\left(t_{n},u\right]}\varphi_{t_{n}}\left(r,\beta_{t_{n}}\right)dF\left(r\right)=\int_{\left(t,u\right]}\varphi_{t}\left(r,\beta_{t}\right)dF\left(r\right)\,\textrm{on }\left\{ t<\tau\right\} \;\mathbf{P}\textrm{-a.s.}
\]
This is also valid for intervals $(t_{n},+\infty)$and
$(t,+\infty)$ (instead of $(t_{n},u]$ and
$(t,u]$). Relation (\ref{eq:star2}) follows immediately.

Let us conclude the first part of the proof by showing that equality
(\ref{eq:star3}) is indeed true.

\vspace{6pt} 
\noindent\textit{Proof of (\ref{eq:star3}). \ }
Recall that $t<t_{n}<u$, $n\in\mathbb{N}$. We start by
noting that the function 
\[
y\mapsto p\left(\frac{r-u}{r-t_{n}}\left(u-t_{n}\right),\, y,\,\frac{r-u}{r-t_{n}}\beta_{t_{n}}\right)
\]
is a density on $\mathbb{R}$ for all $n$. Denoting by $N(\mu,\sigma^{2})$ the normal distribution with expectation $\mu$ and variance $\sigma^{2}$, it follows that the probability measures $N(\frac{r-u}{r-t_{n}}\beta_{t_{n}}(\omega),\frac{r-u}{r-t_{n}}(u-t_{n}))$ converge weakly to $N(\frac{r-u}{r-t}\beta_{t}(\omega),\frac{r-u}{r-t}(u-t))$.
Since the function $g$ is bounded by $M$, we have that
\begin{align}
|G_{t_n,u}(r,\beta_{t_n}(\omega))|&=\left|\int_{\mathbb{R}}g\left(y\right)p\left(\frac{r-u}{r-t_{n}}\left(u-t_{n}\right),\, y,\,\frac{r-u}{r-t_{n}}\beta_{t_{n}\left(\omega\right)}\right)dy\right|\nonumber\\
&\leq M<+\infty\,.\label{eq:HJ1}
\end{align}
Furthermore,
\begin{align}
\lefteqn{\lim_{n\rightarrow\infty} G_{t_n,u}\left(r,\beta_{t_n}(\omega)\right)\nonumber}\\
&=\lim_{n\rightarrow\infty}\int_{\mathbb{R}}g\left(y\right)p\left(\frac{r-u}{r-t_{n}}\left(u-t_{n}\right),\, y,\,\frac{r-u}{r-t_{n}}\beta_{t_{n}}\left(\omega\right)\right)dy\nonumber\\
&=\int_{\mathbb{R}}g\left(y\right)p\left(\frac{r-u}{r-t}\left(u-t\right),\,y,\,\frac{r-u}{r-t}\beta_{t}\left(\omega\right)\right)dy
=G_{t,u}\left(r,\beta_{t}(\omega)\right)\label{eq:HJ2}
\end{align}
immediately follows from the assumption that $g$ is bounded and continuous combined with the weak convergence of the Gaussian measures stated above. Now (\ref{eq:star3}) can be derived using Lebesgue's bounded convergence theorem and the properties of $G_{t_n,u}(r,\beta_{t_n}(\omega))$
and $\varphi_{t_{n}}(r,\beta_{t_{n}}(\omega))$ (and hence $\phi_{t_{n}}(r,\beta_{t_{n}}(\omega))$)
of $\mathbf{P}$-a.s. boundedness and pointwise convergence verified above.

\smallskip
The first part of the proof of the theorem is now finished. In the second part of the proof we consider the case $t=0$ which is divided into two steps. In the first step we assume
that there exists $\varepsilon>0$ such that $\mathbf{P}(\tau>\varepsilon)=1$ and in the second step we will drop this condition. 

Let us assume that there exists $\varepsilon>0$ such that $\mathbf{P}(\tau>\varepsilon)=1$.
Let $(t_{n})_{n\in\mathbb{N}}$ be a decreasing sequence
of strictly positive real numbers converging to $0$: $0<...<t_{n+1}<t_{n},\: t_{n}\downarrow0$ as $n\rightarrow\infty$.
Without loss of generality we assume $t_{n}<\varepsilon$ for all
$n\in\mathbb{N}$. Then (\ref{eq:densitybeta2})
can be rewritten as follows:
\begin{equation}\label{Phi-rewritten}
\phi_{t_{n}}\left(r,\beta_{t_{n}}\right)=\frac{\frac{1}{\sqrt{2\pi t_{n}}}\sqrt{\frac{r}{r-t_{n}}}\exp\left[-\frac{\beta_{t_{n}}^{2}r}{2t_{n}\left(r-t_{n}\right)}\right]}{{\displaystyle \int_{\left(\varepsilon,+\infty\right)}}\frac{1}{\sqrt{2\pi t_{n}}}\sqrt{\frac{s}{s-t_{n}}}\exp\left[-\frac{\beta_{t_{n}}^{2}s}{2t_{n}\left(s-t_{n}\right)}\right]dF\left(s\right)}\,\mathbb{I}_{\left(t_n,+\infty\right)}\left(r\right)\,.
\end{equation}
We have the following auxiliary result:
\begin{lem}
\label{lem:On-the-set-2} Suppose that $\mathbf{P}(\tau>\varepsilon)=1$. Then the function $r\mapsto\phi_{t_{n}}(r,\beta_{t_{n}})$
is $\mathbf{P}_\tau$-a.s. uniformly bounded by some constant $K=K(\varepsilon,\omega)<+\infty$ and, for all $r>0$,
\[
\lim_{n\rightarrow\infty}\phi_{t_{n}}\left(r,\beta_{t_{n}}\right)=1\,,\;\mathbf{P}\textrm{-a.s.}
\]
\end{lem}
\begin{proof}
See Appendix \ref{sec:Proofs-of-Section}.
\end{proof}
Now we turn to the proof of the Markov property of $\beta$ with respect to $\mathbb{F}^\beta$ at $t=0$ under the additional assumption that $\mathbf{P}(\tau>\varepsilon)=1$. Since $\mathbf{E}\left[g(\beta_u)|\mathcal{F}^\beta_0\right]= \lim_{n\rightarrow\infty} \mathbf{E}\left[g(\beta_u)|\mathcal{F}^P_{t_n}\right]$, as in the first part, it is sufficient to verify that
\begin{equation}
\lim_{n\rightarrow\infty}\mathbf{E}\left[g\left(\beta_{u}\right)|\mathcal{F}^P_{t_{n}}\right]=\mathbf{E}\left[g\left(\beta_{u}\right)|\beta_{0}\right],\;\mathbf{P}\textrm{-a.s.}\label{eq:equiv 2}
\end{equation}
Using the formula of total probability, Corollary \ref{cor:LEM if--is} and \eqref{eq:bbdensity}, we can calculate
\begin{align*}
\mathbf{E}\left[g\left(\beta_{u}\right)|\beta_{0}\right]&=\mathbf{E}\left[g\left(\beta_{u}\right)\right]\\
&=g\left(0\right)\,F\left(u\right)+\int_{\left(u,+\infty\right)}\int_{\mathbb{R}}g\left(y\right)p\left(\frac{r-u}{r}u,y,0\right)dy\,dF\left(r\right).
\end{align*}
Consequently, recalling (\ref{eq:equiv1}) for computing
the left-hand side of (\ref{eq:equiv 2}) and the definition of $G_{t_n,u}$ and $G_{t,u}$ by \eqref{eq:prova3} and noting the obvious relation $\lim_{n\rightarrow\infty}\mathbb{I}_{\{ \tau\leq t_{n}\} }=0$,
it is sufficient to prove the following two equalities, $\mathbf{P}$-a.s.:
\begin{align}
\lim_{n\rightarrow\infty}\int_{\left(t_{n},u\right]}\phi_{t_{n}}\left(r,\beta_{t_{n}}\right)dF\left(r\right) &=F\left(u\right)\,,\label{eq:star 2 bis}\\
\lim_{n\rightarrow\infty}\int\limits_{\left(u,+\infty\right)}G_{t_n,u}\left(r,\beta_{t_n}\right)\,\phi_{t_{n}}\left(r,\beta_{t_{n}}\right)dF\left(r\right)&=\int\limits_{\left(u,+\infty\right)}G_{t,u}\left(r,\beta_t\right)\,dF\left(r\right)\,.\label{eq:star 3 bis}
\end{align}
The first equality relies on Lemma \ref{lem:On-the-set-2} and Lebesgue's bounded convergence theorem. For verifying the second equality we use Lebesgue's bounded convergence theorem combined with Lemma \ref{lem:On-the-set-2} and the boundedness and pointwise convergence of the sequence $G_{t_n,u}(\cdot,\beta_{t_n})$ (see (\ref{eq:HJ1}) and (\ref{eq:HJ2})).
 
We now turn to the general case where $\tau>0$. Let $\varepsilon>0$ be arbitrary, but fixed. In what follows the process $^{\varepsilon}\!\beta=(^{\varepsilon}\!\beta_{t},\, t\geq0)$ will denote the process defined by $^{\varepsilon}\!\beta_{t}:=
(\beta_{t}^{r})_{r=\tau\vee\varepsilon}$.
The natural filtration of the process $^{\varepsilon}\!\beta$ will
be denoted by $\mathbb{F}^\varepsilon=(\mathcal{F}^\varepsilon_{t})_{t\geq0}$,
where $\mathcal{F}_{t}^\varepsilon :=\sigma(^{\varepsilon}\!\beta_{s},\,0\leq s\leq t)$.
Obviously, for proving the Markov property of $\beta$ with respect to $\mathbb{F}^\beta$ at $t=0$ it is sufficient to show that $\mathcal{F}^\beta_0:=\mathcal{F}^0_{0+}\vee\mathcal{N}_P$ is $\mathbf{P}$-trivial. In order to show that $\mathcal{F}_{0+}^{0}\vee\mathcal{N}_{P}$ is indeed the trivial $\sigma$-algebra, we consider a set $A\in\mathcal{F}_{0+}^{0}$
and we show that if $\mathbf{P}(A)>0$, then $\mathbf{P}(A)=1$.

If $\mathbf{P}(A)>0$, then there exists an $\varepsilon>0$
such that 
$\mathbf{P}(A\cap\{ \tau>\varepsilon\})>0$.
Since $A\in\mathcal{F}_{0+}^{0}$, it follows that
$A\in\mathcal{F}_{u}^{0}$ for all $0<u\leq\varepsilon$
and, consequently, $A\cap\{ \tau>\varepsilon\} \in\mathcal{F}_{u}^{0}|_{\{ \tau>\varepsilon\} }\vee\mathcal{N}_{P}$ where $\mathcal{F}_{u}^{0}|_{\{ \tau>\varepsilon\}}$ denotes the trace $\sigma$-field of $\mathcal{F}_{u}^{0}$ on the set $\tau>\varepsilon$.
Moreover, on the set $\{\tau>\varepsilon\} $, $\beta_{t}={}^{\varepsilon}\!\beta_{t}$
for all $t\geq0$, i.e., $\beta$ and $^{\varepsilon}\beta$ generate
the same trace filtration on $\{\tau>\varepsilon\}$ and, consequently, $A\cap\{ \tau>\varepsilon\} \in\mathcal{F}_{u}^\varepsilon|_{\{ \tau>\varepsilon\} }\vee\mathcal{N}_{P}$. 
Hence, there exists a set $A_{u}\in\mathcal{F}_{u}^\varepsilon$
such that
\begin{equation}
A\cap\left\{ \tau>\varepsilon\right\} =A_{u}\cap\left\{ \tau>\varepsilon\right\} \label{eq:star_epsilon}
\end{equation}
$\mathbf{P}$-a.s., for all $0<u\leq\varepsilon$. Replacing $u$ by $1/n$, for $n\in\mathbb{N}$ sufficiently large, and defining 
$A_{0}:=\limsup_{n\rightarrow\infty}A_{1/n}$, we obtain that $A_{0}\in\mathcal{F}_{0+}^\varepsilon$.
We know from the first step of the second part of the proof that the $\sigma$-field $\mathcal{F}_{0+}^\varepsilon$
is $\mathbf{P}$-trivial and, consequently, $\mathbf{P}(A_{0})\in\{ 0,1\}$. However, from equality (\ref{eq:star_epsilon}) we have that, $\mathbf{P}$-a.s., $A\cap\{ \tau>\varepsilon\} =A_{0}\cap\{\tau>\varepsilon\}$. By hypothesis we have 
$\mathbf{P}(A\cap\{ \tau>\varepsilon\})>0$
and thus $\mathbf{P}(A_{0} \cap\{\tau>\varepsilon\})=\mathbf{P}(A\cap\{ \tau>\varepsilon\})>0$, which implies that $\mathbf{P}(A_{0})=1$ and, consequently, we get that $\mathbf{P}(A\cap\{ \tau>\varepsilon\})=\mathbf{P}(\{ \tau>\varepsilon\})$.
Since $\varepsilon$ is arbitrary we can take the limit for $\varepsilon\downarrow0$, and we obtain that $\mathbf{P}(A)=\mathbf{P}(\Omega)=1$, which ends the proof.
\end{proof}
\begin{cor}
\label{rem:equalitySTANDARDwithCOMPLETED} The filtration $\mathbb{F}^{P}$
satisfies the usual conditions of right-continuity and completeness.
\end{cor}
\begin{proof}
See, e.g., \cite[(Ch. I, (8.12))]{key-5}. \end{proof}
As a consequence of Corollary \ref{rem:equalitySTANDARDwithCOMPLETED}, the filtrations $\mathbb{F}^{\beta}$ and $\mathbb{F}^{P}$ coincide and, in particular, the $\sigma$-algebra $\mathcal{F}^{\beta}_0$ is $\mathbf{P}$-trivial. 
\begin{rem}
It is worth to mention that in fact the statement of Corollary \ref{rem:equalitySTANDARDwithCOMPLETED} combined with Theorem \ref{TEO:The-information-process} is equivalent to the statement of Theorem \ref{thm:The-process--1}.
\end{rem}
\section{\label{sec:Semimartingale-decomposition} Semimartingale Decomposition of the Information Process}
\noindent This section deals with the semimartingale decomposition of $\beta$ with respect to $\mathbb{F}^{\beta}$. To begin with, we recall the notion of the optional projection of a general measurable process $X$.
\begin{prop}
Let $X$ be a nonnegative measurable process and $\mathbb{F}$ a
filtration satisfying the usual conditions. There exists a unique (up to indistinguishability) $\mathbb{F}$-optional
process $^{o}\!X$ such that 
\[
\mathbf{E}\left[X_{T}\,\mathbb{I}_{\left\{ T<+\infty\right\} }|\mathcal{F}_{T}\right]=\,^{o}\!X_{T}\,\mathbb{I}_{\left\{ T<+\infty\right\} },\;\mathbf{P}\textrm{-a.s.}
\]
for every $\mathbb{F}$-stopping time $T$. 
\end{prop}
\begin{proof}
See, for example, \cite[Ch. IV, (5.6)]{key-3}.
\end{proof}
\begin{defn} (i) \ 
The process $^{o}\!X$ is called the \textit{optional projection} of
the nonnegative measurable process $X$ with respect to $\mathbb{F}$. 

(ii) \ Let $X$ be an arbitrary measurable process. Then
we define the \textit{optional projection $^{o}\!X$ of $X$ with respect
to $\mathbb{F}$} as
\[
^{o}\!X_{t}\left(\omega\right):=\begin{cases}
^{o}\!X_{t}^{+}\left(\omega\right)-\,^{o}\!X_{t}^{-}\left(\omega\right), & \textrm{if }^o\!X_{t}^{+}\left(\omega\right)\wedge ^o\!\!X_{t}^{-}\left(\omega\right)<+\infty,\\
+\infty, & \textrm{otherwise},
\end{cases}
\]
where $^{o}\!X^{+}$ (resp. $^{o}\!X^{-}$) is the optional projection
of the positive part $X^+$ (resp. the negative part $X^{-}$) of $X$ with respect to $\mathbb{F}$.
\end{defn}
\begin{rem}
\label{rem:If--for rem 63} Let $\xi$ be an arbitrary random variable and $\mathcal{G}$ a sub-$\sigma$-field of $\mathcal{F}$. Then the conditional expectations $\mathbf{E}\left[\xi^+|\mathcal{G}\right]$ and $\mathbf{E}\left[\xi^-|\mathcal{G}\right]$ always exist and in analogy to the above definition we agree to define the conditional expectation $\mathbf{E}\left[\xi|\mathcal{G}\right]$ by
\begin{align*}
\mathbf{E}\left[\xi|\mathcal{G}\right]
&:=\begin{cases}
\mathbf{E}\left[\xi^+|\mathcal{G}\right]-\mathbf{E}\left[\xi^-|\mathcal{G}\right], & \textrm{on } \{\mathbf{E}\left[\xi^+|\mathcal{G}\right]\wedge \mathbf{E}\left[\xi^-|\mathcal{G}\right]<+\infty\},\\
+\infty, & \textrm{otherwise}\,.
\end{cases}
\end{align*}
Now let $X$ be an arbitrary measurable process and $\mathbb{F}$ a filtration satisfying the usual conditions. We emphasize that then for every $\mathbb{F}$-stopping time $T$ we have 
$$
\mathbf{E}\left[X_{T}\,\mathbb{I}_{\left\{ T<+\infty\right\} }|\mathcal{F}_{T}\right]=\,^{o}\!X_{T}\,\mathbb{I}_{\left\{ T<+\infty\right\} },\;\mathbf{P}\textrm{-a.s.}
$$
In particular, 
$\mathbf{E}\left[X_{t}|\mathcal{F}_{t}\right]=\,^{o}\!X_{t}$, $\mathbf{P}\textrm{-a.s.}$, for all $t\geq 0$.
\end{rem}
Next we are going to state a slight extension of a well-known result from filtering theory which will be used in the sequel.  The reader may find useful references, e.g., in \cite[Proposition 5.10.3.1]{key-4}, or in \cite[Ch. VI, (8.4)]{key-1000}.

First we will introduce the following definition.
\begin{defn}
Let $B$ be a continuous process, $\mathbb{F}$ a filtration and $T$ an $\mathbb{F}$-stopping time. Then $B$ is called an $\mathbb{F}$-Brownian motion stopped at $T$ if $B$ is an $\mathbb{F}$-martingale with square variation process 
$\langle B,B\rangle$: $\langle B,B\rangle_t=t\wedge T$, $t\geq 0$. 
\end{defn}
\begin{prop}[Innovation Lemma]
\label{pro:innovationLEMMA} Let $\mathbb{F}=(\mathcal{F}_{t})_{t\geq0}$
be a filtration satisfying the usual conditions, $T$ an $\mathbb{F}$-stopping time and $B$ an $\mathbb{F}$-Brownian motion stopped at $T$. Let $Z=(Z_{t},\, t\geq0)$
be an $\mathbb{F}$-optional process such that 
\begin{equation}\label{ass:innovation-lemma}
\mathbf{E}\left[\int_{0}^{t}\left|Z_{s}\right|\, ds\right]<+\infty,\quad t\geq0.
\end{equation}
Let the process $X=(X_{t},\, t\geq0)$ be given by 
\[
X_{t}:=\int_{0}^{t}Z_{s}\, ds+B_{t},\quad t\geq 0\,.
\]
Denote by $^{o}\!Z$ the optional projection of $Z$ with respect to $\mathbb{F}^{X}=(\mathcal{F}_{t}^{X})_{t\geq0}$.
Then the process $b$,
\[
b_{t}:=X_{t}-\int_{0}^{t}{}^{o}\!Z_{s}\, ds,\quad t\geq0,
\]
is an $\mathbb{F}^{X}$-Brownian motion stopped at $T$. 
\end{prop}
\begin{proof}
For the sake of easy reference a proof of this result is provided
in Appendix \ref{sec:Proof of Innovation lemma}.
\end{proof}
In the remainder of this section we shall make use of the filtration $\mathbb{G}=(\mathcal{G}_{t})_{t\geq0}$
defined as 
\begin{equation}
\mathcal{G}_{t}:=\bigcap_{u>t}\mathcal{F}_{u}^{\beta}\vee\sigma\left(\tau\right),\quad t\geq 0\,,\label{eq:FF D}
\end{equation}
which is equal to the initial enlargement of the filtration $\mathbb{F}^{\beta}$ by the $\sigma$-algebra $\sigma(\tau)$. In the sequel,
the process $Z=(Z_{t},\, t\geq0)$ is defined by 
\begin{equation}
Z_{t}:=\frac{\beta_{t}}{\tau-t}\,\mathbb{I}_{\{t<\tau\}}\,,\quad t\geq 0\,. 
\label{eq:processo Z}
\end{equation}
The following auxiliary results will be used to prove the semimartingale
property of the process $\beta$.
\begin{lem}
\label{lem:sqrt-tau-1} We have  
\[
\mathbf{E}\left[\int_{0}^{t}\left|Z_{s}\right|\, ds\right]<+\infty \mbox{ for all } t\geq0\,.
\]
\end{lem}
\begin{proof}
Using the formula of total probability, Corollary \ref{cor:LEM if--is} and \eqref{eq:bbdensity}, we can calculate 
\begin{align*}
\mathbf{E}\left[\int_{0}^{t}\left|Z_{s}\right|\, ds\right] & =
\int_0^{+\infty}\int_0^{t\wedge r}E\left|\beta^r_s\right|/\left(r-s\right)\,ds\,dF\left(r\right)\\
&=(2/\pi)^{1/2}\int_{0}^{+\infty}r^{-1/2}\int_{0}^{t\wedge r}s^{1/2}\left(r-s\right)^{-1/2}\,ds\,dF(r).
\end{align*}
The outer integral on the right-hand side of the above expression
can be split into two integrals, the first one over $(0,t]$ and the second one over $(t,+\infty)$. For the first integral we see that 
\begin{align*}
\lefteqn{\int\limits_{(0,t]}r^{-1/2}\int_{0}^{t\wedge r}s^{1/2}\left(r-s\right)^{-1/2}\,ds\,dF(r)}\\
& \leq\int\limits_{(0,t]}\int_{0}^{t\wedge r}\left(r-s\right)^{-1/2}\,ds\,dF(r)=\int\limits_{(0,t]}2\,r^{1/2}\, dF\left(r\right)\leq2\,t^{1/2}<+\infty\,.
\end{align*}
The second integral can be estimated as follows:
\begin{eqnarray*}
\lefteqn{\int_{\left(t,+\infty\right)}r^{-1/2}\int_{0}^{t\wedge r}s^{1/2}\left(r-s\right)^{-1/2}\,ds\,dF(r)}\\
&\leq& t^{1/2}\int_{\left(t,+\infty\right)}r^{-1/2}\int_{0}^{t}\left(r-s\right)^{-1/2}\, ds\,dF(r)\\
&\leq& t^{1/2}\int_{\left(t,+\infty\right)}r^{-1/2}\,2\,r^{1/2}\,dF\left(r\right)\leq2\,t^{1/2}<+\infty\,.
\end{eqnarray*}
The statement of the lemma is now proved.
\end{proof}
\begin{cor}
\label{cor:The-process-Z} $Z$ defined by (\ref{eq:processo Z})
is integrable with respect to the Lebesgue measure $\mathbf{P}$-a.s.: 
\[
\mathbf{P}\left(\int_{0}^{t}\left|Z_{s}\right|\, ds<+\infty\right)=1, \textrm{ for all } t\geq0\,.
\]
\end{cor}
\begin{rem}
\label{Z-modification}
In view of Corollary \ref{cor:The-process-Z}, there can be found a process $\overline{Z}$ which is indistinguishable from $Z$ such that $\int_0^t\left|\overline{Z}_s\right|\, ds<+\infty$ for all $t\geq 0$ \textit{everywhere} (and not only $\mathbf{P}$-a.s.). Without loss of generality we can assume that $Z$ has this property. If this would not be the case we could modify the paths of $Z$ on a negligible set. By this modification, the optional projection $^o\!Z$ with respect to any filtration $\mathbb{F}$ will stay in the same class of indistinguishable processes. We can also modify  $\beta$ putting $\beta=0$ on the negligible set where the above integrals are not finite for all $t$. In this way the desired property for $Z$ would be fulfilled automatically.
\end{rem} 
\begin{rem}
\label{rem:We-make-the QWERTY} 
(i) \ The process $Z$ is optional with respect to the filtration $\mathbb{G}$ because $Z$ is right-continuous and $\mathbb{G}$-adapted.
 
(ii) \ Let $l_+$ be the Lebesgue measure on $\mathbb{R}_+$. Because of Lemma \ref{lem:sqrt-tau-1}, using Fubini's theorem, there is a measurable subset $\Lambda$ of $\mathbb{R}_+$ such that $l_+(\mathbb{R}_+\setminus\Lambda)=0$ and $E[|Z_t|]<+\infty$ for all $t\in\Lambda$. For later use, we fix a set $\Lambda$ with these properties.

(iii) \ Formulas obtained in Section \ref{sec:Conditional-Expectations} allow to compute the optional projection $^o\!Z$ of $Z$ with respect to the filtration $\mathbb{F}^\beta$ on the set $\Lambda$: For all $t\in\Lambda$ we have $\mathbf{P}$-a.s.
\begin{align*}
^o\!Z_t&=\mathbf{E}\left[Z_{t}|\mathcal{F}_{t}^{\beta}\right] =\mathbf{E}\left[\frac{\beta_{t}}{\tau-t}\,\mathbb{I}_{\left\{ t<\tau\right\} }|\beta_{t}\right]\\
&=\beta_{t}\int_{\left(t,+\infty\right)}\frac{1}{r-t}\phi_{t}\left(r,\beta_{t}\right)dF\left(r\right)\,\mathbb{I}_{\left\{ t<\tau\right\} }\,, 
\end{align*}
where the first equality follows from Remark \ref{rem:If--for rem 63}, the second from the Markov property of the process $\beta$ and Definition (\ref{eq:processo Z}) of the process
$Z$, while the third equality follows directly from Proposition \ref{cor:Let--be-BIS} and the measurability of
$\beta_{t}$ with respect to $\sigma(\beta_{t})$. Note that in the above equality all terms are well-defined for \textit{every} $t\geq 0$, the condition that $t\in\Lambda$ is only needed for the second and third equality.
\end{rem}
\begin{prop}\label{B}
The process $B=(B_{t},\, t\geq0)$ defined by 
\begin{equation}
B_{t}:=\beta_{t}+\int_{0}^{t}Z_{s}\, ds\,,\quad t\geq 0\,,\label{eq:BIGBM}
\end{equation}
is a $\mathbb{G}$-Brownian motion stopped at $\tau$.
\end{prop}
\begin{proof}
Note that by Corollary \ref{cor:The-process-Z} and Remark \ref{Z-modification}, the process $(Z_{t},\, t\geq0)$
is integrable with respect to the Lebesgue measure, hence $B$ is
well-defined. 

It is clear that the process $B$ is continuous and
$\mathbb{G}$-adapted. In order to prove that it is indeed a $\mathbb{G}$-Brownian motion stopped at $\tau$, it suffices to prove that the process $B$ and the process $X$ defined by $X_t:=B_{t}^{2}-(t\wedge\tau)$, $t\geq0$, are both $\mathbb{G}$-martingales. To this end we shall use Corollary
\ref{cor:LEM if--is}.

First we show that $B$ is a $\mathbb{G}$-martingale.
Let $n\in\mathbb{N}$ be an arbitrary but fixed natural number,
$0<t_{1}<...<t_{n-1}<t_{n}:=t,\: h\geq0$ and $g$ an arbitrary bounded
Borel function. Recalling the definition of $B^r$ in \eqref{eq:BMfrom_extBB}, we have that
\begin{align*}
\lefteqn{\mathbf{E}\left[\left(B_{t+h}-B_{t}\right)g\left(\beta_{t_{1}},\dots,\beta_{t_{n}},\tau\right)\right]}\\
&=\int_{\left(0,+\infty\right)}\mathbf{E}\left[\left(B_{t+h}-B_{t}\right)g\left(\beta_{t_{1}},...,\beta_{t_{n}},\tau\right)|\tau=r\right]dF\left(r\right)\\
 & =\int_{\left(0,+\infty\right)}\mathbf{E}\left[\left(B_{t+h}^{r}-B_{t}^{r}\right)g\left(\beta_{t_{1}}^{r},\ldots,\beta_{t_{n}}^{r},r\right)\right]dF(r)=0,
\end{align*}
because $B^{r}$ is a Brownian motion and hence a martingale with respect to the filtration generated by $\beta^{r}$ (cf. Lemma \ref{BMr}). Using a monotone class argument, from this it can easily be derived that $B$ is a martingale with respect to $\mathbb{G}$. 

It remains to prove that the process $X$ is a $\mathbb{G}$-martingale. Putting $X_{t}^{r}:=(B_{t}^{r})^{2}-(t\wedge r)$, $t\geq0$,
and repeating the same arguments used above, we see that 
\begin{align*}
\lefteqn{\mathbf{E}\left[\left(X_{t+h}-X_{t}\right)g\left(\beta_{t_{1}},\dots,\beta_{t_{n}},\tau\right)\right]}\\
&=\int_{\left(0,+\infty\right)}\mathbf{E}\left[\left(X_{t+h}-X_{t}\right)g\left(\beta_{t_{1}},\ldots,\beta_{t_{n}},\tau\right)|\tau=r\right]dF\left(r\right)\\
 & =\int_{\left(0,+\infty\right)}\mathbf{E}\left[\left(X_{t+h}^{r}-X_{t}^{r}\right)g\left(\beta_{t_{1}}^{r},\ldots,\beta_{t_{n}}^{r},r\right)\right]dF(r)=0,
\end{align*}
since $X^{r}$ is a martingale with respect to the filtration generated by $\beta^{r}$. As above from this follows that $X$ is a martingale with respect to $\mathbb{G}$. This completes the proof of Lemma \ref{B}.
\end{proof}
We are now ready to state the main result of this section.
\begin{thm}
\label{TEO:semimartFb} The process $b=(b_{t},\, t\geq0)$ given by
\begin{align}
\nonumber b_{t}&:=\beta_{t}+\int_{0}^{t}\mathbf{E}\left[\frac{\beta_{s}}{\tau-s}\,\mathbb{I}_{\{s<\tau\}}\big|\beta_{s}\right]\, ds\\
&=\beta_t+\int_0^t \beta_s \int_{(s,+\infty)}\frac{1}{r-s}\,\phi_s(r,\beta_s)\, dF(r)\,\mathbb{I}_{\{s<\tau\}}\, ds\,,
\label{eq:semimart-dec-1}
\end{align}
where the second equality holds $\mathbf{P}$-a.s.,
is an $\mathbb{F}^{\beta}$-Brownian motion stopped at $\tau$. Thus the information process $\beta$ is an $\mathbb{F}^\beta$-semimartingale whose decomposition is determined by \eqref{eq:semimart-dec-1}.
\end{thm}
\begin{proof}
First we notice that by Lemma \ref{lem:sqrt-tau-1} and Remarks \ref{rem:If--for rem 63} and \ref{rem:We-make-the QWERTY} (iii), the integrands of the second and third term of \eqref{eq:semimart-dec-1} are well-defined for all $s\geq 0$ and they are equal $l_+$-a.e. and integrable on $[0,t]\times\Omega$ with respect to $l_+\times\mathbf{P}$. This yields the second equality $\mathbf{P}$-a.s. Now we can apply Proposition \ref{pro:innovationLEMMA} to the processes $(-Z)$ and $X$ where $X$ is defined by $X_t=\int_{0}^{t}(-Z_{s})\,ds+B_{t}$, $t\geq0$. By Proposition \ref{B} we know that the process $B$ with $B_{t}:=\beta_{t}+\int_{0}^{t}Z_{s}\, ds$, $t\geq 0$, is a $\mathbb{G}$-Brownian motion stopped
at $\tau$ (see Proposition \ref{B}). Note that $X=\beta$. According to Proposition \ref{pro:innovationLEMMA} the process $b$ with 
$$
b_t=X_t-\int_0^{t}{^o(-Z)}_s\,ds=\beta_t+ \int_0^{t}{^o\!Z}_s\,ds=\beta_t+\int_{0}^{t}\mathbf{E}\left[\frac{\beta_{s}}{\tau-s}\,\mathbb{I}_{\{s<\tau\}}\big|\beta_{s}\right]\, ds\,,
$$
is a Brownian motion stopped at $\tau$, with respect to $\mathbb{F}^X=\mathbb{F}^\beta$,  where we have used Remark \ref{rem:We-make-the QWERTY} (iii) for the third equality. This completes the proof of Theorem \ref{TEO:semimartFb}.
\end{proof}
\begin{rem}
\label{rem:QUA COV 1} Note that the quadratic variation $\langle \beta,\beta\rangle$ of the information process $\beta$ is given by
\[
\langle \beta,\beta\rangle _{t}=\langle b,b\rangle _{t}=t\wedge\tau,\quad t\geq0\,.
\]
This follows immediately from the semimartingale decomposition \eqref{eq:semimart-dec-1} of $\beta$, see, for example, \cite[Ch. IV, (1.19)]{key-3} (the quadratic variation does not depend on the filtration, provided that the process is a continuous semimartingale).
\end{rem}
\section{\label{sec:Local-time-of} Example: pricing a Credit Default Swap}
\noindent A rather common financial contract that is traded in the credit market is the Credit Default Swap (CDS). A CDS with maturity $T>0$ is a contract between a buyer who wants to protect against the possibility that the default of a financial asset will take place before $T$, and a seller who provides such insurance. If the default has not occurred before $T$, the buyer will pay to the seller a fee until the maturity. But if the default time $\tau$ occurs before the maturity, the fee will be paid until the default and then the seller will give immediately to the buyer a pre-established amount of money $\delta$, called \textit{recovery}. The recovery may depend on the time at which the default occurs and, hence, it is modeled by a positive function $\delta:[0,T]\rightarrow\mathbb{R}_{+}$.
We follow the approach developed in \cite{key-29}, and the reader can find some details also in \cite[Ch. 7.8]{key-4}. We assume that the default-free spot interest rate $r$ is constant and that the fee that the buyer must pay to
the seller is paid continuously in time according to some rate $\kappa>0$, that is to say, the buyer has to pay an amount $\kappa dt$ during the time $dt$ until $\tau\wedge T$. In case the default time $\tau$ occurs before the maturity $T$, the seller will pay to the buyer a recovery $\delta(\tau)$ at time $\tau$. If the pricing
measure is $\mathbf{P}$ and the market filtration is $\mathbb{G}=(\mathcal{G}_{t})_{\geq0}$, the price $S_{t}(\kappa,\delta,T,r)$ at time $t$ of the CDS is given by
\begin{equation}
S_{t}\left(\kappa,\delta,r,T\right):=e^{rt}\,\mathbf{E}\left[e^{-r\tau}\delta\left(\tau\right)\mathbb{I}_{\left\{ t<\tau\leq T\right\} }-\int_{t\wedge\tau}^{T\wedge\tau}e^{-rv}\kappa \, dv|\mathcal{G}_{t}\right].\label{eq:CDS}
\end{equation}
We would like to make a comparison between the result obtained in
our model with the one presented in \cite{key-29}. However, differently from \cite{key-29}, we reduce ourselves to the simple situation where the market filtration can be either the minimal filtration $\mathbb{H}$ that makes $\tau$ a stopping time or the filtration $\mathbb{F}^{\beta}$ generated by the information process $\beta$. It is worth to note that the minimal filtration that makes $\tau$ a stopping time is of particular importance in the theory of enlargement of filtrations and its applications to mathematical models of credit risk. We refer to the series of papers \cite{key-23,key-24,key-25} and to the book \cite{key-4} for the topics of mathematical finance, where the filtration $\mathbb{H}$ is used, and to the books \cite{key-18} and \cite{key-7} for a discussion of the subject from a purely mathematical point of view. 

In the following, let us make the further assumption that $r=0$.
We first recall the pricing formula for the market filtration $\mathbb{G}=\mathbb{H}$.
\begin{prop}
If the market filtration $\mathbb{G}$ is the minimal filtration $\mathbb{H}=(\mathcal{H}_{t})_{t\geq0}$ that makes $\tau$ a stopping time, then the price $S_{t}(\kappa,\delta,0,T)$ at time $t$ of a CDS is given by 
\begin{align}
S_{t}\left(\kappa,\delta,0,T\right) & =\mathbf{E}\left[\delta\left(\tau\right)\mathbb{I}_{\left\{ t<\tau\leq T\right\} }-\kappa\left(\left(\tau\wedge T\right)-t\right)\mathbb{I}_{\left\{t<\tau\right\} }|\mathcal{H}_{t}\right]\nonumber \\
 & =\mathbb{I}_{\left\{t<\tau\right\} }\frac{1}{G\left(t\right)}\left(-\int_{t}^{T}\delta\left(v\right)dG\left(v\right)-\kappa\int_{t}^{T}G\left(v\right)dv\right)\label{eq:CDSeasy}
\end{align}
where $G(u):=\mathbf{P}(u<\tau)=1-F(u)$
is supposed to be strictly greater than 0 for every $u\in[0,T]$.
\end{prop}
\begin{proof}
See \cite[Lemma 2.1]{key-29}.
\end{proof}
\vspace{-5pt}
In our model, i.e., if  the market filtration is the filtration $\mathbb{F}^{\beta}=(\mathcal{F}_{t}^{\beta})_{t\geq0}$
generated by the information process $\beta$, the pricing formula
is given in the following proposition.
\begin{prop}
\label{lem:If-,-i.e.,}If $\mathbb{G}=\mathbb{F}^{\beta}$, then the
price $S_{t}(\kappa,\delta,0,T)$ at time $t$ of a CDS
is given by
\vspace{-5pt}
\begin{align}
S_{t}\left(\kappa,\delta,0,T\right) & =\mathbf{E}\left[\delta\left(\tau\right)\mathbb{I}_{\left\{ t<\tau\leq T\right\} }-\kappa\left(\left(\tau\wedge T\right)-t\right)\mathbb{I}_{\left\{t<\tau\right\} }|\mathcal{F}_{t}^{\beta}\right]\label{price-formula} \\
 & =\mathbb{I}_{\left\{t<\tau\right\} }\left(-\int_{t}^{T}\delta\left(v\right)d_{v}\Psi_{t}\left(v\right)-\kappa\int_{t}^{T}\Psi_{t}\left(v\right)dv\right)\label{eq:CDSmy}
\end{align}
where $\Psi_{t}(u):=\mathbf{P}(u<\tau|\mathcal{F}_{t}^{\beta}) =\int_{u}^{+\infty}\phi_{t}(v,\beta_{t})\,dF(v)$ and the writing $d_{v}\Psi_{t}(v)$ in the above formula means that the integral is computed using $v$ as integrating variable.
\end{prop}
\vspace{-10pt}
\begin{proof}
Concerning the first term in \eqref{price-formula}, in view of Corollary \ref{LEM:prop_cond_exp_tau} and \eqref{eq:densitybeta2}, we have,  $\mathbf{P}$-a.s.:
\begin{align*}
\mathbf{E}\left[\delta\left(\tau\right)\mathbb{I}_{\left\{ t<\tau\leq T\right\} }|\mathcal{F}_{t}^{\beta}\right]&=\mathbb{I}_{\left\{t<\tau\right\} }\int_{t}^{T}\delta\left(v\right)\phi_{t}\left(v,\beta_{t}\right)dF\left(v\right)\\
&=\mathbb{I}_{\left\{t<\tau\right\} }\int_{t}^{T}\delta\left(v\right)d_{v}\Phi_{t}\left(v\right)
\end{align*}
where $\Phi_{t}(v):=\mathbf{P}(\tau\leq v|\mathcal{F}_{t}^{\beta})=1-\Psi_{t}(v)$ and hence
\begin{align}\label{Ex1}
\mathbf{E}\left[\delta\left(\tau\right)\mathbb{I}_{\left\{ t<\tau\leq T\right\} }|\mathcal{F}_{t}^{\beta}\right]&=
\mathbb{I}_{\left\{t<\tau\right\} }\int_{t}^{T}\delta\left(v\right)d_{v}\Phi_{t}\left(v\right)\\
&=-\mathbb{I}_{\left\{t<\tau\right\} }\int_{t}^{T}\delta\left(v\right)d_{v}\Psi_{t}\left(v\right)\,.
\end{align}
Concerning the second term in \eqref{price-formula}, again in view of Corollary \ref{LEM:prop_cond_exp_tau} and \eqref{eq:densitybeta2}, we obtain $\mathbf{P}$-a.s.,
\begin{eqnarray}
\nonumber\mathbf{E}\left[T\wedge\tau|\mathcal{F}_{t}^{\beta}\right]\mathbb{I}_{\left\{t<\tau\right\} } &=&\left(\int_{t}^{T}vd_{v}\Phi_{t}\left(v\right)+T\left(1-\Phi_{t}\left(T\right)\right)\right)\mathbb{I}_{\left\{t<\tau\right\} }\\
\nonumber&=&\left(-\int_{t}^{T}vd_{v}\Psi_{t}\left(v\right)+T\Psi_{t}\left(T\right)\right)\mathbb{I}_{\left\{t<\tau\right\}}\\
&=&\left(\int_{t}^{T}\Psi_{t}\left(v\right)dv+t\Psi_{t}\left(t\right)\right)\mathbb{I}_{\left\{t<\tau\right\}}\,. \label{Ex2}
\end{eqnarray}
Inserting \eqref{Ex1} and \eqref{Ex2} into the price formula \eqref{price-formula} and noting that $\Psi_t(t)=1$ on $\{t<\tau\}$ $\mathbf{P}$-a.s., we obtain the asserted result \eqref{eq:CDSmy}. The proof of the proposition is completed. 
\end{proof}
Although there is a formal analogy between \eqref{eq:CDSeasy} and \eqref{eq:CDSmy}, the pricing formulas are quite different. The second formula \eqref{eq:CDSmy} is much more informative because it uses the observation of $\beta_t$ and the Bayes estimate of $\tau$, namely, the a posteriori distribution of $\tau$ after observing $\beta_t$. Given that $t<\tau$, the price in \eqref{eq:CDSeasy}  is a deterministic value, while in \eqref{eq:CDSmy} the price depends on the observation $\beta_t$.
\begin{rem}
The so-called \textit{fair spread} of a CDS at time $t$ is the value $\kappa^{*}$ such that $S_{t}(\kappa^{*},\delta,0,T)=0$.
In our model we have that
\[
\kappa^{*}=-\frac{\int_{t}^{T}\delta\left(r\right)d_{r}\Psi_{t}\left(r\right)}{\int_{t}^{T}\Psi_{t}\left(r\right)dr}
\]
while, in the simpler situation where the market filtration is $\mathbb{H}$,
the fair spread is given by
\[
\kappa^{*}=-\frac{\int_{t}^{T}\delta\left(r\right)dG\left(r\right)}{\int_{t}^{T}G\left(r\right)dr}.
\]
\end{rem}
When dealing with problems related to credit risk, it is of interest to consider the case where the market filtration is a filtration $\mathbb{G}$ obtained by progressively enlarging a reference filtration $\mathbb{F}$ with another filtration, $\mathbb{D}=(\mathcal{D}_{t})_{t\geq0}$,
which is responsible for modeling the information associated with the default time $\tau$. Traditionally, the filtration $\mathbb{D}$
has been settled to be equal to the filtration $\mathbb{H}$ generated
by the single-jump process occurring at $\tau$. We intend
to consider a different setting where the reference filtration $\mathbb{F}$ will be enlarged with the filtration $\mathbb{F}^{\beta}$ and we will provide the relative pricing formulas of credit instruments like the CDS. However, this is beyond the scope of the present paper.
\begin{appendix}
\section{The Bayes Formula}
\noindent Here the basic results on the Bayes formula are recalled without proofs. For further details we refer to \cite[Ch. II.7.8]{key-8} or any book on Bayesian Statistics.

Let $\tau$ and $X$ be random variables on a probability space $(\Omega,\,\mathcal{F},\,\mathbf{P})$ with values in measurable spaces $(E_{1},\,\mathcal{E}_{1})$
and $(E_{2},\,\mathcal{E}_{2})$, respectively. Let $\mathbf{P}_{r}$
be a regular conditional distribution of $X$ with respect to $\tau=r$,
i.e., for $B\in\mathcal{E}_{2}$, $\mathbf{P}_{r}(B)=\mathbf{P}(X\in B|\tau=r)$,
$\mathbf{P}_{\tau}$-a.s. By $\mathbf{P}_{\tau}$ we denote the distribution of $\tau$ on $(E_{1},\,\mathcal{E}_{1})$ (called the \textit{a priori} distribution). Moreover, for
$C\in\mathcal{E}_{1}$, let $\mathbf{G}_{C}$ be defined as follows:
\begin{align}
\mathbf{G}_{C}\left(B\right) & :=\int_{C}\mathbf{P}_{r}\left(B\right)\mathbf{P}_{\tau}\left(dr\right),\, B\in\mathcal{E}_{2}\,.\label{eq:starUAN}
\end{align}
We are interested in the \textit{a posteriori} probability $\mathbf{Q}(x,C):=\mathbf{P}(\tau\in C|X=x)$,
for $x\in E_{2}$ and $C\in\mathcal{E}_{1}$. By $\mathbf{P}_{X}$ we denote the law of $X$.
\begin{thm}[Bayes' Theorem]
We have $\mathbf{G}_{C}\ll\mathbf{P}_{X}$ and 
\[
\mathbf{Q}\left(x,C\right)=\frac{d\mathbf{G}_{C}}{d\mathbf{P}_{X}}\left(x\right),\quad x\in E_{2},\; C\in\mathcal{E}_{1},\;\mathbf{P}_{X}\textrm{-a.s.}
\]
\end{thm}
Now we assume that there exists a $\sigma$-finite measure $\mu$
on $(E_{2},\,\mathcal{E}_{2})$ such that $\mathbf{P}_{r}\ll\mu$, for all $r\in E_1$. Furthermore, we assume that there is a measurable function $p$ on $(E_{1}\times E_{2},\,\mathcal{E}_{1}\otimes\mathcal{E}_{2})$
such that
\[
p\left(r,x\right)=\frac{d\mathbf{P}_{r}}{d\mu}\left(x\right),\quad \mu\textrm{-a.e.}, \; r\in E_{1}\,.
\]
\begin{lem}
We have that
\begin{enumerate}
\item[(i)] \ $\mathbf{G}_{C}\ll\mu$ and 
$\displaystyle{\frac{d\mathbf{G}_{C}}{d\mu}(x)=\int_{C}p(r,x)\mathbf{P}_{\tau}(dr),\;\mu\textrm{-a.e.}}$,
\item[(ii)] \ $\mathbf{P}_{X}\ll\mu$ and 
$\displaystyle{\frac{d\mathbf{P}_{X}}{d\mu}(x)=\int_{E_{1}}p(r,x)\mathbf{P}_{\tau}(dr),\;\mu\textrm{-a.e.}}$
\end{enumerate}
\end{lem}
\begin{cor}[Bayes' Formula]
\label{cor:(Bayes-formula)-a.s.}
\[\mathbf{P}\left(\tau\in C|X=x\right)=
\mathbf{Q}\left(x,C\right)=\frac{{\displaystyle \int_{C}p\left(r,x\right)\mathbf{P}_{\tau}\left(dr\right)}}{{\displaystyle \int_{E_{1}}p\left(v,x\right)\mathbf{P}_{\tau}\left(dv\right)}},\quad \mathbf{P}_{X}\textrm{-a.s.},\; C\in\mathcal{E}_{1}\,.
\]
\end{cor}
\begin{cor}[A Posteriori Density]
\label{cor:The-conditional-density} The a posteriori density $q(r,x)$ of $\tau$ under $X=x$ with respect to $\mathbf{P}_{\tau}$ is given by
\[
q\left(r,x\right):=p\left(r,x\right) \left(\int_{E_{1}}p\left(v,x\right)\mathbf{\mathbf{P}_{\tau}}\left(dv\right)\right)^{-1},\quad r\in E_1,\; x\in E_2\,.
\]
Thus 
\[
\mathbf{P}\left(\tau\in C|X=x\right)=\int_{C}q\left(r,x\right)\mathbf{\mathbf{P}_{\tau}}\left(dr\right),\quad \mathbf{P}_{X}\textrm{-a.s.},\; C\in\mathcal{E}_{1}\,.
\]
\end{cor}
For the proof of Theorem \ref{lem:PRE-theo-4.3}, we have to choose  $(E_1,\mathcal{E}_1)=(\mathbb{R}_+,\mathcal{B}(\mathbb{R}_+))$, $(E_2,\mathcal{E}_2)=(\mathbb{R},
\mathcal{B}(\mathbb{R}))$, $X=\beta_t$, $\tau$ as the default time and the $\sigma$-finite measure $\mu$ as the measure $\delta_0+l$ where $\delta_0$ is the Dirac measure at $0$ and $l$ is the Lebesgue measure on $\mathbb{R}$. Then Corollary \ref{cor:(Bayes-formula)-a.s.} can be applied to the second term on the right-hand side of \eqref{prop:HJ1second} which yields the second term on the right hand side of \eqref{eq:condexptau-1}. 
\section{\label{sec:Proof of Innovation lemma} Proof of Proposition \ref{pro:innovationLEMMA} (Innovation Lemma)}
\noindent We start by observing that the asssumption \eqref{ass:innovation-lemma} of Proposition \ref{pro:innovationLEMMA} implies
\[
\int_{0}^{t}\left|Z_{s}\right|\, ds<+\infty,\quad t\geq0, \;\mathbf{P}\textrm{-a.s.}
\]
Putting $Z$ equal to zero on the negligible set where  $\int_0^t |Z_s|\, ds=+\infty$ for some $t\geq0$, we can assume without loss of generality that this property holds everywhere. Hence $\int_{0}^{t}Z_{s}\, ds$ is well defined and finite everywhere.

By Remark \ref{rem:If--for rem 63} we know that for the optional projection $^o\!Z$ of $Z$ with respect to $\mathbb{F}^X$ we have that $^{o}\!Z_{s}=\mathbf{E}\left[Z_{s}|\mathcal{F}_{s}^{X}\right]$, $\mathbf{P}$-a.s., for all $s\geq0$. This yields $\mathbf{E}\left[\int_{0}^{t}\left|^{o}\!Z_{s}\right|\, ds\right]\leq \mathbf{E}\left[\int_{0}^{t}\left|Z_{s}\right|\, ds\right]<+\infty$, $t\geq0$, and hence $\int_{0}^{t}|{^o\!Z}_{s}|\, ds<+\infty$, for all $t\geq0$, $\mathbf{P}$-a.s. As above, without loss of generality we can modify $^{o}\!Z$ on a $\mathbf{P}$-negligible set such that $\int_{0}^{t}\left|^{o}\!Z_{s}\right|\, ds<+\infty$, $t\geq0$, everywhere, while being still the optional projection of $Z$. Hence $\int_{0}^{t}\,^{o}\!Z_{s}\, ds$ is well-defined and finite everywhere, for all $t\geq0$. Since $^o\!Z$ is $\mathbb{F}^X$-optional, it is clear that the process $\int_{0}^{\cdot}\,^{o}\!Z_{s}\, ds$ is $\mathbb{F}^{X}$-adapted. Therefore, the process $b$ with $b_t:=X_t-\int_0^t{^o\!Z}_s\, ds$, $t\geq0$, is $\mathbb{F}^{X}$-adapted. Furthermore, $Z$ being $\mathbb{F}$-optional, the process $\int_{0}^{\cdot}Z_{s}\, ds$ is $\mathbb{F}$-adapted and hence the process $X$ is also $\mathbb{F}$-adapted. This yields the inclusion $\mathbb{F}^X\subseteq\mathbb{F}$.

Next we show that the process $b$ is an $\mathbb{F}^{X}$-martingale. Obviously, $b_{t}$ is integrable for all $t\geq0$ and, as shown above, $b$ is  $\mathbb{F}^{X}$-adapted. Let $0\leq s<t$. For showing the martingale property, using Fubini's theorem, we first notice that there exists a Borel set $\Lambda\subseteq\mathbb{R}_+$ such that $l_+(\mathbb{R}_+\setminus \Lambda)=0$ and $Z_s$, and hence also $^o\!Z_s$, is integrable for all $s\in\Lambda$. Now 
\begin{align*}
\lefteqn{\mathbf{E}\left[b_{t}-b_{s}|\mathcal{F}_{s}^{X}\right]}\\
&=\mathbf{E}\left[B_{t}-B_{s}|\mathcal{F}_{s}^{X}\right]+\mathbf{E}\left[\int_{s}^{t}\left(Z_{u}-{}^{o}\!Z_{u}\right)du|\mathcal{F}_{s}^{X}\right]\\
&=\mathbf{E}\left[\mathbf{E}\left[B_{t}-B_{s}|\mathcal{F}_{s}\right]|\mathcal{F}_{s}^{X}\right]+\mathbf{E}\left[\int_{s}^{t}\mathbb{I}_{\Lambda}\left(u\right)\,\left(Z_{u}-{}^{o}\!Z_{u}\right)du|\mathcal{F}_{s}^{X}\right]\\
 & =\int_{0}^{t}\mathbb{I}_{\Lambda}\left(u\right)\,\mathbf{E}\left[Z_{u}-{}^{o}\!Z_{u}|\mathcal{F}_{s}^{X}\right]du\\
&=\int_{0}^{t}\mathbb{I}_{\Lambda}\left(u\right)\,\mathbf{E}\left[\mathbf{E}\left[Z_{u}-{}^{o}\!Z_{u}|\mathcal{F}^X_u\right]|\mathcal{F}_{s}^{X}\right]du=0
\end{align*}
where we have used that $\mathbb{F}^X\subseteq\mathbb{F}$ and that $B$ is an $\mathbb{F}$-martingale, Fubini's theorem and properties of the optional projection (see Remark \ref{rem:If--for rem 63}). This proves that $b$ is a continuous $\mathbb{F}^X$-martingale. 

Finally, since $B$ is an $\mathbb{F}$-Brownian motion stopped at $T$, from the definition of $b$ it is clear that $b$ is a continuous $\mathbb{F}$-semimartingale with square variation process $\langle b,b\rangle$: $\langle b,b\rangle_t=t\wedge T$, $t\geq 0$. In view of the well-known fact that the square variation of continuous semimartingales does not depend on the filtration, this is also true with respect to $\mathbb{F}^X$. This proves that $b$ is an $\mathbb{F}^X$-Brownian motion stopped at $T$. 
\section{\label{sec:Proofs-of-Section} Proof of Lemma \ref{lem:On-the-set-2}}
\noindent First we recall that $(t_{n})_{n\in\mathbb{N}}$
is a strictly decreasing sequence converging to $0$ such that 
$0<t_{n+1}<t_{n}<\ldots t_{1}<\varepsilon,\, t_{n}\downarrow0$.
For proving that $\lim_{n\rightarrow\infty}\phi_{t_{n}}(r,\beta_{t_{n}})=1$, using the assumption that $\mathbf{P}(\tau>\varepsilon)=1$, \eqref{eq:densitybeta2} and \eqref{eq:bbdensity}, we can write
\begin{align}
\nonumber\lefteqn{\phi_{t_{n}}\left(r,\beta_{t_{n}}\right)}\\
=&\frac{\left(2\pi t_{n}\right)^{-1/2}\,{r}^{1/2}\left(r-t_{n}\right)^{-{1/2}}\,\exp\big[-\beta_{t_{n}}^{2}r/2t_{n}\left(r-t_{n}\right)\big]\mathbb{I}_{\left(t_n,+\infty\right)}\left(r\right)}{\!\!\!\!\int\limits_{\left(\varepsilon,+\infty\right)}\left(2\pi t_{n}\right)^{-1/2}\,{s}^{1/2}\left(s-t_{n}\right)^{-1/2}\, \exp\big[-\beta_{t_{n}}^{2}s/2t_{n}\left(s-t_{n}\right)\big]\, dF\left(s\right)}\\
=&\frac{r^{1/2}\left(r-t_{n}\right)^{-1/2}\; \mathbb{I}_{\left(t_n,+\infty\right)}\left(r\right)}{\int\limits_{\left(\varepsilon,+\infty\right)}{s}^{1/2}\left(s-t_{n}\right)^{-1/2}\exp\big[\beta_{t_{n}}^{2}\left(s-r\right)/2\left(r-t_{n}\right)\left(s-t_{n}\right)\big]\,dF\left(s\right)}\,.\label{eq:boh}
\end{align}
First we note that for $s\in\left(\varepsilon,+\infty\right)$  
\begin{equation}\label{Estimate:square-root}
1\leq s^{1/2}\left(s-t_{n}\right)^{-1/2}\leq\varepsilon^{1/2} \left(\varepsilon-t_{n}\right)^{-1/2}\leq\varepsilon^{1/2}\left(\varepsilon-t_{1}\right)^{-1/2}\,.
\end{equation}
Secondly, if $s\in(\varepsilon,r)$, then
$\exp\left[\beta_{t_{n}}^{2}(s-r)/2(r-t_{n})(s-t_{n})\right]\leq1,$
while for $s\in[r,+\infty)$ we have that $(s-r)/(s-t_{n})\leq1$,
and thus $$
\exp\left[\beta_{t_{n}}^{2}(s-r)/2(r-t_{n})(s-t_{n})\right]\leq\exp\left[\beta_{t_{n}}^{2}/2(r-t_{n})\right].
$$
We note that the right-hand side of this inequality is bounded with respect to $n$, since $t_n\downarrow0$ and $\beta_{t_n}\rightarrow0$. 
Furthermore, 
\[
\lim_{n\rightarrow\infty}s^{1/2}\left(s-t_{n}\right)^{-1/2}\exp\left[\beta_{t_{n}}^{2}\left(s-r\right)/2\left(r-t_{n}\right)\left(s-t_{n}\right)\right]=1\,.
\]
Thus we can apply Lebesgue's bounded convergence theorem and it follows
that
\[
\lim_{n\rightarrow\infty}
\int\limits_{\left(\varepsilon,+\infty\right)}\!\!\!\!s^{1/2}\left(s-t_{n}\right)^{-1/2}\exp\left[\beta_{t_{n}}^{2}\left(s-r\right)/2\left(r-t_{n}\right)\left(s-t_{n}\right)\right]\,dF\left(s\right)=1\,.
\]
Finally, as for the numerator in \eqref{eq:boh}, 
$$
\lim_{n\rightarrow\infty}r^{1/2}(r-t_{n})^{-1/2}\, \mathbb{I}_{(t_n,+\infty)}(r)=1,\; r>0\,,
$$ 
we have $\lim_{n\rightarrow\infty}\phi_{t_{n}}(r,\beta_{t_{n}})=1$.

In order to prove the $\mathbf{P}_\tau$-a.s. uniform boundedness of $\phi_{t_{n}}(\cdot,\beta_{t_{n}})$, first we notice that in view of  \eqref{Estimate:square-root} the numerator in \eqref{eq:boh} is uniformly bounded in $r\in(\varepsilon,+\infty)$ and using the assumption that $\mathbf{P}(\tau>\varepsilon)=1$ it follows that it is $\mathbf{P}_\tau$-a.s. uniformly bounded in $r\in(0,+\infty)$, too. It remains to verify that the denominator in \eqref{eq:boh} is bounded from below by a strictly positive constant only depending on $\varepsilon$ and $\omega$. Using \eqref{Estimate:square-root}, for $s\in(\varepsilon,+\infty)$ the integrand in the denominator of \eqref{eq:boh} can be estimated from below by
\begin{eqnarray*}
\lefteqn{{s}^{1/2}\left(s-t_{n}\right)^{-1/2}\exp\big[\beta_{t_{n}}^{2}\left(s-r\right)/2\left(r-t_{n}\right)\left(s-t_{n}\right)\big]}\\
&\geq&\exp\big[\beta_{t_{n}}^{2}\left(s-r\right)/2\left(r-t_{n}\right)\left(s-t_{n}\right)\big]\\
&\geq&\exp\left[-\beta_{t_{n}}^{2}/\left(\varepsilon-t_{1}\right)\right]\, \mathbb{I}_{\left(\varepsilon,r\right)}+\mathbb{I}_{\left[r,+\infty\right)}\geq e^{-\gamma}
\end{eqnarray*}
where $\gamma=\gamma(\varepsilon,\omega)=\sup_{n\geq1}\beta_{t_{n}}^{2}(\omega)/(\varepsilon-t_{1})<+\infty$. Hence for the denominator in \eqref{eq:boh} it follows
\begin{eqnarray*}
\lefteqn{e^{-\gamma}=\int_{\left(\varepsilon,+\infty\right)}e^{-\gamma}\, dF\left(s\right)}\\
&\leq&\int_{\left(\varepsilon,+\infty\right)}{s}^{1/2}\left(s-t_{n}\right)^{-1/2}\exp\big[\beta_{t_{n}}^{2}\left(s-r\right)/2\left(r-t_{n}\right)\left(s-t_{n}\right)\big]\,dF\left(s\right)\,.
\end{eqnarray*}
The proof of Lemma \ref{lem:On-the-set-2} is completed.
\end{appendix}
\subsection*{Acknowledgment}
\noindent This work has been financially supported by the European Community's FP 7 Program under contract PITN-GA-2008-213841, Marie Curie ITN \flqq Controlled Systems\frqq.

\vspace{0.4in}
\noindent Matteo Ludovico Bedini, Intesa Sanpaolo, Milano, Italy;\\
e-mail: matteo.bedini@intesasanpaolo.com

\vspace{0.2in}
\noindent Rainer Buckdahn, Laboratoire de Math\'ematiques CNRS-UMR 6204,\\ Universit\'e de Bretagne Occidentale, Brest, France; School of Mathematics, Shandong University, Jinan, Shandong Province, P.R.China;\\
e-mail: rainer.buckdahn@univ-brest.fr 

\vspace{0.2in}
\noindent Hans-J\"urgen Engelbert, Friedrich-Schiller-Universit\"at Jena, Jena, Germany; e-mail: hans-juergen.engelbert@uni-jena.de
\end{document}